\documentclass[reqno,10pt]{amsart}
\usepackage{amsmath,amsthm,amscd,amssymb,latexsym,upref,enumerate,hyperref}

\usepackage{epsfig}
\usepackage[T2A,OT1]{fontenc}
\usepackage[ot2enc]{inputenc}
\usepackage[russian,english]{babel}

\usepackage{xcolor}

\numberwithin{equation}{section}

\newcommand{\bbD}{{\mathbb{D}}}

\newcommand{\bbR}{{\mathbb{R}}}

\newcommand{\bbC}{{\mathbb{C}}}

\newcommand{\bbT}{{\mathbb{T}}}

\newcommand{\cA}{{\mathcal{A}}}
\newcommand{\cB}{{\mathcal{B}}}
\newcommand{\cC}{{\mathcal{C}}}

\newcommand{\cJ}{{\mathcal{J}}}

\newcommand{\cM}{{\mathcal{M}}}

\newcommand{\cP}{{\mathcal{P}}}
\newcommand{\cQ}{{\mathcal{Q}}}
\newcommand{\cR}{{\mathcal{R}}}

\newcommand{\cV}{{\mathcal{V}}}

\newcommand{\fA}{{\mathfrak{A}}}

\newcommand{\fa}{{\mathfrak{a}}}

\newcommand{\fc}{{\mathfrak{c}}}

\newcommand{\sE}{{\mathsf{E}}}

\newcommand{\SL}{{\mathrm{SL}}}
\newcommand{\SU}{{\mathrm{SU}}}

\renewcommand{\a}{\alpha}

\newcommand{\e}{{\epsilon}}
\newcommand{\s}{{\sigma}}

\renewcommand{\phi}{\varphi}

\newcommand{\pd}{{\partial}}

\renewcommand{\Re}{\text{\rm Re}\,}
\renewcommand{\Im}{\text{\rm Im}\,}

\newcommand{\tr}{\text{\rm trace}\,}

\newcommand{\diam}{\text{\rm diam}}
\newcommand{\adj}{\text{\rm adj}}

\newtheorem{theorem}{Theorem}[section]

\newtheorem{lemma}[theorem]{Lemma}
\newtheorem{proposition}[theorem]{Proposition}

\newtheorem{corollary}[theorem]{Corollary}

\theoremstyle{definition}
\newtheorem{definition}[theorem]{Definition}

\newtheorem{remark}[theorem]{Remark}

\title
[Reflectionless canonical systems, I]{Reflectionless canonical systems, I. \\ Arov gauge and right limits}
\author{Roman Bessonov, Milivoje Luki\'c, and Peter Yuditskii}

\address{
\begin{flushleft}
Roman Bessonov: bessonov@pdmi.ras.ru\\\vspace{0.1cm}
St.\,Petersburg State University\\  
Universitetskaya nab. 7-9, 199034 St.\,Petersburg, RUSSIA\\
\vspace{0.1cm}
St.\,Petersburg Department of Steklov Mathematical Institute\\ Russian Academy of Sciences\\
Fontanka 27, 191023 St.Petersburg,  RUSSIA
\end{flushleft}
}
\address{
\begin{flushleft}
Milivoje Luki\'c: milivoje.lukic@rice.edu\\\vspace{0.1cm}
Rice University, Department of Mathematics MS-136, 
Houston, TX 77251-1892, USA.
\end{flushleft}
}
\address{
\begin{flushleft}
Peter Yuditskii: Petro.Yudytskiy@jku.at\\\vspace{0.1cm}
Abteilung f\"ur Dynamische Systeme und Approximationstheorie, Johannes Kepler Universit\"at Linz, A-4040 Linz, Austria
\end{flushleft}
}

\begin{document}

\begin{abstract}
In spectral theory, $j$-monotonic families of $2\times 2$ matrix functions appear as transfer matrices of many one-dimensional operators. We present a general theory of such families, in the perspective of canonical systems in Arov gauge. This system resembles a continuum version of the Schur algorithm, and allows to restore an arbitrary Schur function along the flow of associated boundary values at infinity. In addition to results in Arov gauge, this provides a gauge-independent perspective on the Krein--de Branges formula and the reflectionless property of right limits on the absolutely continuous spectrum. This work has applications to inverse spectral problems which have better behavior with respect to a normalization at an internal point of the resolvent domain. 
\end{abstract}

\maketitle

\section{Introduction}
For the most famous classes of one-dimensional self-adjoint and unitary operators, such as Schr\"odinger, Dirac, Jacobi, and CMV operators \cite{LevSar,SimonSzego}, a basic object in their study is a $2\times 2$ transfer matrix $\cA(z,t)$, which describes the evolution of formal eigensolutions at ``energy" (spectral parameter) $z\in \bbC$ from values at $0$ to values at $t$. These are sometimes viewed as part of families $\cA(z,t_1,t_2)$ with $\cA(z,t) = \cA(z,0,t)$ and the cocycle condition
\[
\cA(z,t_1,t_2) \cA(z,t_2,t_3) = \cA(z,t_1,t_3)
\]
(we use the convention that matrices act as linear operators by right-multiplication on row vectors). The exact construction of the transfer matrix includes some conventions specific to each class of operators, but in all cases, the standard conventions give matrices which are entire functions of $z$, obey
\[
\det \cA(z,t) = 1, \qquad \forall z, t
\]
and
\[
\cA(z,0) = I, \qquad \forall z.
\]
These assumptions will hold throughout our paper. Our primary motivation here are continuous systems, so we think of $t$ as a continuous parameter and $\cA(z,t)$ will be continuous in $t$.

Another crucial property, first observed by Weyl \cite{Weyl} in the setting of Schr\"odinger operators, is the nesting property of certain Weyl disks and an associated limit point/limit circle dichotomy. This is naturally expressed in the following language.

\begin{definition}
Let $j$ be a $2\times 2$ matrix such that $j=j^* = j^{-1}$. An entire $2\times 2$ matrix valued function $\cA(z)$ is called $j$-inner if it obeys $j - \cA(z) j \cA(z)^* \ge 0$ for $z \in \bbC_+ = \{z\in \bbC: \; \Im z > 0\}$ and $j - \cA(z) j \cA(z)^* = 0$ for $z \in \bbR$.
\end{definition}

There is a reflection symmetry: for every $j$-inner entire function $\cA(z)$,
\begin{equation}\label{reflectionsymmetry0}
\cA(\bar z)^* = j \cA(z)^{-1} j, \qquad \forall z \in \bbC. 
\end{equation}
 This holds on $\bbR$ by $j\cA(z) j \cA(z)^* = I$, and  
 since both sides are entire, it holds on $\bbC$.

\begin{definition}
A family of matrix functions $\cA(z,t)$ parametrized by a real parameter $t$ is called $j$-monotonic if $\cA(z,t_1)^{-1} \cA(z,t_2)$ is $j$-inner whenever $t_1 < t_2$.
\end{definition}

To the family $\cA(z,t)$ we associate for $z \in \bbC_+$, $t \ge 0$ the Weyl disks
\begin{equation}\label{WeylDiskeqn}
D(z,t) = \{ w \mid \begin{pmatrix} w & 1 \end{pmatrix} \cA(z,t) j \cA(z,t)^* \begin{pmatrix} w & 1 \end{pmatrix}^* \ge 0 \}.
\end{equation}
The $j$-monotonic property implies
\[
\cA(z,t_1) j \cA(z,t_1)^* \ge \cA(z,t_2) j \cA(z,t_2)^*, \qquad z \in \bbC_+, \quad t_1 < t_2
\]
and precisely corresponds to the nesting property of the Weyl disks, $D(z,t_2) \subset D(z,t_1)$ for $t_1 < t_2$.  

The exact $j$-monotonic property depends on the class of operators; for some classes of self-adjoint operators one has $\cJ$-monotonicity with
\begin{equation}\label{cJmatrixdefn}
\cJ = \begin{pmatrix} 0 & -i \\ i & 0 \end{pmatrix}.
\end{equation}
However, this matrix can always be changed by a suitable conjugation of the transfer matrices, so without loss of generality, in this paper we will study $j$-monotonic transfer matrices for the choice
\[
j = \begin{pmatrix}
-1 & 0 \\
0 & 1
\end{pmatrix}.
\]
Note that with this convention, $\cA(z,0) = I$ implies $D(z,0) = \overline{\bbD}$, i.e., our Weyl disks are subsets of the unit disk. Thus, our choice of $j$ also provides a convenient compactification. Moreover, for $z\in \bbR$, the values $\cA(z,t)$ belong to the group of $2\times 2$ matrices
\[
\SU(1,1) = \{ U \mid Uj U^* = j\text{ and } \det U = 1\}.
\]
It follows from definitions that Weyl disks are not affected by a transformation
\begin{equation}\label{gaugetransformation}
\cA(z,t) \mapsto \cA(z,t) U(t), \qquad U(t) \in \SU(1,1).
\end{equation}
Borrowing terminology from Yang--Mills theory, we say that the family of transfer matrices has a \emph{gauge freedom} with the gauge group $\SU(1,1)$, while the Weyl disks are \emph{observables}, i.e., gauge independent. In direct spectral theory, transfer matrices are often normalized by a certain behavior at $\infty$, and in the theory of canonical systems, many works have been written in what we will call Potapov--de Branges gauge (PdB gauge), which is normalized at $z=0$  \cite{P60,dB,AD08,Remling}. However, a different gauge was suggested  by  D.Z. Arov,  see \cite{DEY} for historical remarks:

\begin{definition}
We say the $j$-monotonic family $\cA(z,t)$ is:
\begin{itemize}
\item in Potapov--de Branges gauge (PdB-gauge) if $\cA(0,t) = I$ for every $t$;
\item in Arov gauge (A-gauge) if for every $t$,
\begin{equation}\label{A}
\cA(i,t) = \begin{pmatrix}a_{11}&0\\a_{21}&a_{22}\end{pmatrix}, \qquad a_{11} >0, \quad a_{22}>0.
\end{equation}
\end{itemize}
\end{definition}

Every $j$-monotonic family can be uniquely placed in PdB-gauge, by choosing $U(t) = \cA(0,t)^{-1}$ in \eqref{gaugetransformation}, and that transformation preserves continuity. Likewise, every $j$-monotonic family can be uniquely transformed into A-gauge by \eqref{gaugetransformation}, and the transformation preserves continuity:

\begin{proposition}\label{l07}
Every $j$-monotonic family can be uniquely placed in Arov gauge via a transformation \eqref{gaugetransformation}. Moreover, if the family $\cA(z,t)$ is continuous, so is the corresponding family $U(t)$.
\end{proposition}

A dual perspective is given by canonical systems:

\begin{definition}
A canonical system is an initial value problem of the form
\begin{equation}\label{canonicalsystemgeneral}
\cA(z,t) j = j + \int_0^t \cA(z, \xi) \left( i z \cP(\xi) - \cQ(\xi)\right) d\nu(\xi), \qquad z \in \bbC,
\end{equation}
with $2\times2$ matrix valued coefficients $\cP, \cQ$ locally integrable with respect to a positive continuous Borel measure $\nu$ on $\bbR$ and such that $\cP \ge 0$, $\cQ = - \cQ^*$, $\tr(j\cP) = \tr(j\cQ) = 0$ $\nu$-almost everywhere. 
In particular, we say the canonical system \eqref{canonicalsystemgeneral} is:
\begin{itemize}
\item in Potapov--de Branges gauge (PdB-gauge) if $\cQ = 0$ $\nu$-a.e.,
\item in Arov gauge (A-gauge) if $\cP + \cQ$ is lower triangular and $\tr\cQ = 0$ $\nu$-a.e..
\end{itemize}
\end{definition}

Note that \eqref{canonicalsystemgeneral} is the integral form of the initial value problem 
\begin{equation}\label{canonicalsystemgeneraldifferentialform}
(\partial_\nu \cA)(z,t) j = \cA(z,t) ( iz \cP(t) - \cQ(t)) , \qquad  \cA(z,0) = \begin{pmatrix}1&0\\0&1\end{pmatrix},
\end{equation}
under the regularity assumptions that $\cA$ is locally absolutely continuous with respect to $\nu$, $\partial_\nu \cA$ denotes the Radon--Nikodym derivative with respect to $\nu$, and equality is assumed $\nu$-almost everywhere. By standard Volterra arguments, this initial value problem has a unique solution on any interval on which $\nu$ is finite; we call this solution the transfer matrix of the canonical system.

\begin{remark}
The classical  Dirac operator corresponds to the canonical system \eqref{canonicalsystemgeneraldifferentialform} with $\cP(t)=I$, and for a one-dimensional Schr\"odinger operator with potential $q(t)$, transfer matrices solve the canonical system \eqref{canonicalsystemgeneraldifferentialform} with
$$
\cP(t)=\frac 1 2\begin{pmatrix} 1&1\\1&1
\end{pmatrix},\quad 
\cQ(t)=\frac i 2\begin{pmatrix} q(t)-1&q(t)+1\\ q(t)+1&q(t)-1
\end{pmatrix}.
$$
In both cases $\nu(t)=t$. These gauges are not universal: not any canonical system can be reduced to one of these forms. In particular, one of them can not be reduced to the other by a gauge transform \eqref{gaugetransformation}.
\end{remark}

For canonical systems in A-gauge, it follows from the definition that $\cP+ \cQ$ is a nonnegative scalar multiple of a matrix of the form
\begin{equation}\label{AplusBformula1}
A+B = \begin{pmatrix} 1 & 0 \\ - 2 \fa & 1 \end{pmatrix}.
\end{equation}
By grouping that multiplier with the measure and changing the notation, we write canonical systems in A-gauge in the form
\begin{equation}\label{canonicalsystemArov}
\fA(z,\ell) j = j + \int_0^\ell \fA(z, l) \left( i z A(l) - B(l)\right) d\mu(l), \qquad z \in \bbC.
\end{equation}
where $A \ge 0$ and $B^* = -B$, so \eqref{AplusBformula1} gives
\begin{equation}\label{ABfromfa}
A=	\begin{pmatrix} 1& - \overline{\fa}\\
-{\fa}&1
\end{pmatrix},\quad B=
\begin{pmatrix} 0&  \overline{\fa}\\
- \fa&0
\end{pmatrix}.
\end{equation}
and $A \ge 0$ implies $\lvert \fa \rvert \le 1$. Due to this representation, the pair $(\mu,\fa)$ where $\mu$ is a positive measure which is finite on compact subsets of $\bbR$, and $\fa \in L^\infty(d\mu)$ with $\lVert \fa \rVert_\infty \le 1$, are the parameters of the canonical system in A-gauge.

In this paper we will denote by the same letter a measure $\mu$, finite on compact subsets of $\bbR$, and the corresponding distribution function which is an increasing function. Thus, we can write, e.g., $\mu([\ell_1, \ell_2)) = \mu(\ell_2) - \mu(\ell_1)$ for $\ell_1 \le \ell_2$.  It is well-known that each non-decreasing continuous function on $\bbR$ 
 is the distribution function of some nonnegative continuous Borel measure $\mu$ finite on compacts in $\bbR$.

We will explain in Section~\ref{section2} that our definitions for $j$-monotonic families and for canonical systems are precisely compatible. In particular, the solution of any canonical system in A-gauge is a continuous $j$-monotonic family $\fA(z,\ell)$ in A-gauge with $\det \fA(z,\ell) = 1$ and $\fA(z,0) = I$. Conversely:

\begin{theorem}\label{thmfamilytocansystem}
Let $\fA(z,\ell)$ be a continuous $j$-monotonic family in A-gauge with $\det \fA(z,\ell) = 1$ for all $z, \ell$ and $\fA(z,0) = I$. Then:
\begin{enumerate}[(a)]
\item $\fA(z,\ell)$ is the solution of a canonical system in A-gauge \eqref{canonicalsystemArov};
\item The matrix entry $\fA_{22}(i,\ell)$ is a decreasing function of $\ell$ and the positive measure $\mu$ is determined by its distribution function 
\[
\mu(\ell) = \log \fA_{11}(\ell) = - \log \fA_{22}(i,\ell).
\]
\item 
The function $ \fA(i,\ell)$ is a.c. with respect to $\mu$ and the parameters $\fa$ are determined by 
\[
A(\ell) + B(\ell) = \begin{pmatrix} 1 & 0 \\ -2\fa(\ell) & 1 \end{pmatrix} = - (\fA(i,\ell))^{-1} \partial_\mu \fA(i,\ell) j, \qquad \mu\text{-a.e. }\ell.
\]
Moreover, $\lvert \fa(\ell) \rvert \le 1$ for $\mu$-a.e.\ $\ell$.
\item at $z=i$, the solution has the form
		\begin{equation}\label{Agaugematrixati}
		\fA(i,\ell) = \begin{pmatrix}
		e^{\mu(\ell)} & 0 \\
		-e^{\mu(\ell)} \kappa(\ell) & e^{-\mu(\ell)}
		\end{pmatrix}, 
		\end{equation}
		where
		\begin{equation}\label{kappaformula}
		\kappa(\ell) = \int_0^\ell 2 \fa(l) e^{-2\mu(l)} d\mu(l).
		\end{equation}
\end{enumerate}
\end{theorem}

Among the conclusions of Theorem~\ref{thmfamilytocansystem}, note that it follows from \eqref{canonicalsystemArov} that, if $\mu([\ell_1,\ell_2)) = 0$ for some $\ell_1 < \ell_2$, then $\fA(z,\ell)$ is constant on $\ell \in [\ell_1,\ell_2]$. In particular, $\mu$ is strictly increasing if and only if the family $\fA$ is, in a sense, strictly $j$-monotonic.

One of the goals of this paper is a systematic consideration of $j$-monotonic families/canonical systems in A-gauge. Part of the reason is that we consider this to be a natural point of view, which should be apparent from our proofs of certain general facts previously known in PdB gauge. For instance, since the Weyl disks are nested, they shrink as $\ell$ increases to a point or a disk. This alternative is independent of the choice of $z \in \bbC_+$, and determined in A-gauge completely by $\mu$:

\begin{proposition}\label{prop18}
 Consider a canonical system in A-gauge \eqref{canonicalsystemArov} and the associated Weyl disks.
\begin{enumerate}[(a)]
\item (``limit circle" case) If $\sup_{\ell} \mu(\ell) < \infty$, then $\cap_{\ell > 0} D(z,\ell)$ is a closed disk (with nonzero diameter) for every $z \in \bbC_+$.
\item (``limit point" case) If $\sup_{\ell} \mu(\ell) = \infty$, then $\cap_{\ell > 0} D(z,\ell)$ consists of a single point for every $z \in \bbC_+$.
\end{enumerate}
\end{proposition}

In the limit point case, the intersection defines a function
\[
\{ s_+(z)\} = \bigcap_{\ell > 0} D(z,\ell)
\]
which is a Schur function, i.e., an analytic function $s_+ : \bbC_+ \to \overline{\bbD}$.

Until now, we were vague about the interval of definition of the $j$-monotonic family/canonical system, as the results hold on any interval. From now on, we always assume the canonical system is defined at least on the half-line $[0,\infty)$, and assume that it is in the limit point case. In other words, $\mu(\ell) \to \infty$ as $\ell \to \infty$. It is also convenient to assume that the $j$-monotonic family is not constant on any interval of $\ell$; this corresponds to assuming that $\mu$ gives positive measure to any interval.

Coefficient stripping can be implemented on the canonical system by the usual truncation of the parameters, or directly on the $j$-monotonic family: coefficient stripping by length $\ell > 0$ is the operation
\[
\{ \fA(z,l) \}_{l \in [0,\infty)} \mapsto \{ \fA(z,\ell)^{-1} \fA(z,l + \ell) \}_{l \in [0,\infty)}.
\]
The coefficient stripped transfer matrices have Schur functions $s_+(z,\ell)$.

\begin{proposition}[Ricatti equation] \label{propRicatti}
For an arbitrary canonical system in A-gauge, for any $z \in \bbC_+$, the family of Schur functions $s_+(z,\ell)$ is absolutely continuous with respect to $\mu$ and
\begin{equation}\label{Ricattiequation}
\partial_\mu s_+(z,\ell) = \begin{pmatrix} s_+(z,\ell) & 1 \end{pmatrix} (-iz A(\ell) + B(\ell) ) \begin{pmatrix} 1 \\ s_+(z,\ell) \end{pmatrix}.
\end{equation}
\end{proposition}

Conversely, the boundary values at $\infty$ of the family of Schur functions determine the parameters $\fa$:

\begin{proposition} \label{cor218}
For $\mu$-a.e. $\ell > 0$, the Schur functions have the nontangential boundary values
\begin{equation}\label{boundarySchur}
\lim_{\substack{ z\to \infty \\ \arg z \in [\delta,\pi-\delta]}} s_+(z,\ell) = \frac{ \fa(\ell) }{ 1 +\sqrt{1 - \lvert \fa(\ell) \rvert^2 }}, \qquad \forall \delta  > 0.
\end{equation}
\end{proposition}

Note that the right-hand side of \eqref{boundarySchur} determines $\fa(\ell)$ uniquely; if we denote this right-hand side by $\fc(\ell)$, we have the mutually inverse formulas
\[
\fc(\ell) = \frac{\fa(\ell)}{1+ \sqrt{1 - \lvert \fa(\ell) \rvert^2}}, \qquad \fa(\ell) = \frac{2\fc(\ell)}{1+\lvert \fc(\ell) \rvert^2}
\]
which correspond to a continuous bijection from $\overline{\bbD}$ to itself. For $\fa(\ell)\in\bbD$ this corresponds to the following matrix identity
\[
\sqrt{\cV(\fa(\ell))}=\cV(\fc(\ell)),\qquad
\cV(a) := \frac 1{\sqrt{1-\lvert a \rvert^2}} \begin{pmatrix} 1 & - \bar a \\ -a & 1 \end{pmatrix}, \quad a \in\bbD.
\]

Using A-gauge, we give a new proof of the de Branges mean type theorem (see Section 39 in \cite{dB}) discovered by Krein \cite{Krein97}. 
\begin{theorem} \label{thjul10}
For an arbitrary canonical system \eqref{canonicalsystemgeneral}, the exponential type of the transfer matrix
\begin{equation}\label{9jul420}
\sigma(t):=\limsup_{z\to\infty}\frac{\log\|\cA(z,t)\|}{|z|}
\end{equation}
can be computed as 
\begin{equation}\label{9jul4}
\sigma(t)=\int_0^t\sqrt{\det \cP(\tau)}\,d\nu(\tau).
\end{equation}
\end{theorem}

In particular, in A-gauge \eqref{9jul4} is of the form
\[
\sigma(\ell)=\int_0^\ell\sqrt{1-|\fa(l)|^2}d\mu(l).
\]

\begin{remark}
In discrete systems, the parameter $\ell$ usually corresponds to a polynomial degree; likewise, in some continuous systems $\ell$ corresponds to the exponential type of $\fA(z,\ell)$, i.e., $\s(\ell)=\ell$, and $\mu$ is absolutely continuous w.r.t. to $\ell$, see e.g. \cite{BLY2, DEY}. In this case, $\mu$ is expressed in terms of the continuous Verblunsky parameter as 
\[
 \mu(\ell)=\int_0^\ell\frac{d l}{\sqrt{1-|\fa(l)|^2}}.
\]
\end{remark}

There are two associated existence results. One is the existence of a canonical system on an interval $[0,L]$ corresponding to a prescribed $j$-inner function $\cA(z,L)$; this inverse problem was the main result of Potapov \cite{P60}, and was proved in a much more general setting. Another inverse problem, proved by de Branges \cite{dB}, is the existence of a canonical system on $[0,\infty)$ with prescribed spectral Schur function $s_+$. Of course, these are not the only ways of obtaining canonical systems. In \cite{BLY2}, we directly construct reflectionless canonical systems in A-gauge and prove that they have the desired spectral functions, and we do not use the abstract existence results.

It is common in the theory of canonical systems to parametrize the line in a special way; in PdB-gauge, the standard parametrization involves taking $\tr \cP = 1$ and setting $\nu$ to be Lebesgue measure. However, if $j$-monotonicity is viewed as the central notion, it is more natural to formulate the theory up to a reparametrization of the real line; moreover, in direct spectral theory, the parametrization is already fixed by the underlying operator structure (and usually corresponds to the exponential type of the transfer matrix of order one or $1/2$); finally, in inverse spectral theory, it will be natural to use a parametrization which provides a linearization of the shift under the generalized Abel-Jacobi mapping \cite{BLY2}. For these reasons, our formulations don't fix a parametrization. Our formulation for canonical systems offers the same flexibility: if $g: [0,\infty) \to [0,\infty)$ is a monotone bijection, then the reparametrization
\begin{equation}\label{reparametrization}
\tilde \mu(\ell) = \mu(g(\ell)), \qquad \tilde \fa(\ell) = \fa(g(\ell))
\end{equation}
affects the solution by $\tilde \fA(z,\ell) = \fA(z,g(\ell))$. In particular, observables like the Weyl disks and the Schur function $s_+(z)$ do not change. By a deep result of de Branges, translated into A-gauge, this is the only nonuniqueness:

\begin{theorem}[de Branges uniqueness theorem in A-gauge] \label{theoremdeBrangesuniqueness}
If two canonical systems in A-gauge with parameters $(\mu,\fa)$ and $(\tilde \mu ,\tilde \fa)$ have the same spectral function $s_+$, then there exists a monotone bijection $g: [0,\infty) \to [0,\infty)$ such that \eqref{reparametrization} holds.
\end{theorem}

Canonical systems can also be considered in a two-sided setting (parametrized by $\ell \in \bbR$), with the same definition of PdB-gauge and A-gauge. A reflection of the real line, which preserves $j$-monotonicity (see Section 3), is given by
\begin{equation}\label{canonicalsystemreflectiongeneral}
\{ \cA(z,t) \}_{t\in\bbR} \mapsto \{ j_1 \cA(z,-t) j_1 \}_{t\in\bbR}, \qquad j_1 = \begin{pmatrix} 0 & 1 \\ 1 & 0 \end{pmatrix}.
\end{equation}
This allows us to define $s_-$ as the Schur function which corresponds to the system $\{ j_1 \cA(z,-t) j_1 \}_{t \ge 0}$. However, the operation \eqref{canonicalsystemreflectiongeneral} does not preserve A-gauge. An additional factor will appear; we denote
\[
v(z) = \frac{z-i}{z+i}, \qquad z \in \bbC.
\]

\begin{proposition} \label{lemmanegativehalflineAgauge}
For a full-line canonical system in A-gauge $\{\fA(z,\ell)\}_{\ell\in\bbR}$ with coefficients $(\mu,\fa)$, its Schur function $s_-$ is given by $s_-(z) = v(z) \overleftarrow{s_+}(z)$, where $\overleftarrow{s_+}$ denotes the canonical system in A-gauge
\begin{equation}\label{3jun1}
\overleftarrow{\fA}(z,\ell) =  \begin{pmatrix} v(z) & 0 \\ 0 & 1 \end{pmatrix} j_1 \fA(z,-\ell) j_1 \begin{pmatrix}v(z) & 0 \\ 0  & 1 \end{pmatrix}^{-1}
\end{equation}
This system has A-gauge parameters $\overleftarrow{\mu}(\ell) = - \mu(-\ell)$, $\overleftarrow{\fa}(\ell) = \overline{\fa(-\ell)}$.
\end{proposition}

In summary, a two-sided canonical system in A-gauge is encoded by two Schur functions $s_\pm$, which are arbitrary   up to the single normalization condition $s_-(i) = 0$. In \cite{BLY2} this serves as a natural compactification condition.

This paper is also the first part of a series in which we develop a general theory of reflectionless one-dimensional systems and the interplay between the reflectionless property and almost periodicity of parameters of the canonical system. The reflectionless property is a certain pseudocontinuation relation between the two spectral functions which encode the two half-line restrictions of the operator; in the setting for Schur functions, the canonical system is \emph{reflectionless on $\sE$} if
\[
\overline{s_+(\xi+i0)} = s_-(\xi + i0), \qquad \text{a.e. }\xi \in\sE.
\]
It was first observed as a property of periodic operators and finite gap quasiperiodic operators; by Kotani theory \cite{Kot82}, it is a general feature of ergodic operators with zero Lyapunov exponent on the spectrum. Remling proved that the reflectionless property is a general property of right limits of operators with absolutely continuous spectrum, in the setting of Jacobi matrices \cite{Rem11} and Schr\"odinger operators \cite{Rem07}, with an extension to canonical systems in PdB-gauge by Acharya \cite{Ach16}.  In the current paper, we extend Remling's theorem to canonical systems in a gauge-independent setting. As a corollary, almost periodicity of canonical system coefficients implies the reflectionless property on the a.c. spectrum:

\begin{theorem}\label{t001new}
Assume that for all $L>0$, the functions
\[
\int_{t}^{t+L} \cP(\tau)\,d\nu(\tau), \quad \int_{t}^{t+L} \cQ(\tau)\,d\nu(\tau)
\]
are uniformly almost periodic matrix functions of $t$. Then the canonical system \eqref{canonicalsystemgeneral} is reflectionless on its a.c.\ spectrum $\{\xi \mid \lvert s_+(\xi+i0) \rvert < 1 \} \cup \{ \xi \mid \lvert s_-(\xi +i0) \rvert < 1\}$.
\end{theorem}

In particular, Theorem~\ref{t001new} applies in A-gauge, and applies to Dirac operators.

This is the first paper in our theory of reflectionless canonical systems. In the second paper \cite{BLY2}, we consider the opposite direction, and construct reflectionless canonical systems with a Dirichlet-regular Widom spectrum $\sE$ with the DCT property. 
For this class of spectra $\sE$, we prove that reflectionless canonical systems in A-gauge always have almost periodic parameters, but that this is not always true in PdB-gauge, nor in a gauge normalized at $\infty$. For that problem, A-gauge is the correct general setting, since it corresponds to a normalization with respect to a point which is always an internal point of the resolvent domain.  In this sense, the current paper also serves as a foundation for the construction in \cite{BLY2}.

\bigskip
\noindent
\textbf{Acknowledgment.}  The work of R.B. in Sections 6,7 is supported by grant RScF 19-11-00058 of the Russian Science Foundation. In the rest of the paper,
M.L. was supported in part by NSF grant DMS--1700179 and 
P.Y. was supported by the Austrian Science Fund FWF, project no: P32885-N.

\section{Canonical systems in Arov gauge and $j$-monotonic families} \label{section2}

In this section, we prove the basic facts about $j$-monotonic families and canonical systems from the introduction. We denote by $\SL(2,\bbC)$ the set of all $2\times2$ complex matrices with unit determinant.  A $2\times 2$ matrix $T$ is called
\begin{itemize}
	\item $j$-expanding, if $T j T^* - j \ge 0$;
	\item $j$-unitary, if $T j T^* - j = 0$;
	\item $j$-contractive, if $T j T^* - j \le 0$.
\end{itemize}

\begin{lemma} \label{lemmaAgaugematrixi}
A lower triangular matrix $T  \in \SL(2,\bbC)$ with positive diagonal coefficients is $j$-contractive if and only if it is of the form
\[
T = \begin{pmatrix} \lambda & 0 \\ h & \lambda^{-1} \end{pmatrix}
\]
 for some $\lambda \ge 1$ and $\lvert h \rvert \le \lambda - 1/\lambda$.  Moreover, it is $j$-unitary if and only if $T = I$. 
\end{lemma}

\begin{proof}
$T$ is obviously of this form for some $\lambda > 0$ and $h \in \bbC$. Then
\begin{equation}\label{TjTcalc}
j-T j T^* = \begin{pmatrix} \lambda^2 - 1
  & \lambda \bar h \\
\lambda h  & 1 - \lambda^{-2} + \lvert h\rvert^2
\end{pmatrix}.
\end{equation}
This matrix is positive definite if and only if $\lambda^2 - 1 \ge 0$ and $\det (j-T j T^*) \ge 0$. Since $\det (j-T j T^*) = (1-\lambda^{-2})(\lambda^2-1) - \lvert h\rvert^2$, the criterion for $j$-contractivity follows. The criterion for $j$-unitarity is obvious from \eqref{TjTcalc}.
\end{proof}

We can now describe the gauge transformation of an arbitrary $j$-monotonic family into A-gauge:

\begin{proof}[Proof of Prop.~\ref{l07}]
Let us start with a $j$-contractive matrix $T\in \SL(2,\bbC)$. Since $TjT^* \le j$,  for the vector $(a, b) = (1 , 0)T$ we have 
\[
-|a|^2 + |b|^2 = (a, b)j(a, b)^* \le (1, 0)j(1, 0)^* = -1.
\]	
In particular, $|a|>|b|$ and we can define
\begin{equation}\label{4nov201}
U = \frac 1{\sqrt{ \lvert a \rvert^2 - \lvert b \rvert^2 }} \begin{pmatrix} \bar a & - b \\ - \bar b &  a \end{pmatrix}
\end{equation}
It is straightforward to verify that $U \in \SU(1,1)$ and that $(1,0)TU = (\lambda, 0)$ with $\lambda = \sqrt{ \lvert a \rvert^2 - \lvert b \rvert^2 }>0$. Then $TU$ is in lower-triangular form and $(TU)_{11} > 0$. Since $\det(TU) = \det T = 1$, the diagonal entries of $TU$ are positive.

Note that the construction of $U$ depends continuously on $T$. If this construction is applied to $T = \cA(i,t)$, it gives a one-parameter family $U(t) \in \SU(1,1)$. If $\cA(z,t)$ is continuous, so is $U(t)$.

Assume that $T$ is $j$-contractive with $\det T = 1$ and $U_1, U_2 \in \SU(1,1)$ are such that $TU_1$, $TU_2$ are lower triangular with positive diagonal terms. Then $U = (TU_1)^{-1} (TU_2) = U_1^{-1} U_2$ has the same property and is $j$-unitary. By applying Lemma~\ref{lemmaAgaugematrixi}  we conclude $U_1 = U_2$.
\end{proof}

Now let us take the perspective of canonical systems.  The solution of any canonical system is a continuous $j$-monotonic family. Indeed, for $t_1\le t_2$,
\[
\cA(z, t_1) j \cA(z, t_1)^{*} - \cA(z, t_2) j \cA(z, t_2)^{*} =
2\Im z \int_{t_1}^{t_2}\cA(z,\tau)\cP(\tau)\cA(z,\tau)^*\,d\nu(\tau)
\]
which implies that $\cA(z, t_1) j \cA(z, t_1)^{*} \ge \cA(z, t_2) j \cA(z, t_2)^{*}$ for $z \in \bbC_+$ with equality for $z \in \bbR$. Since $\tr(\cP j) = \tr(\cQ j) = 0$, the property $\det \cA(z,t) = 1$ for all $z, t$ follows from the formula for the derivative of the determinant  $\partial_\nu \det \cA(z,t) = \tr (\adj \cA(z,t) \partial_\nu \cA(z,t))$, where $\adj$ denotes the adjugate matrix.

The solution $\cA(z,t)$ of the canonical system will be called the transfer matrix, and the definitions of PdB-gauge and A-gauge for canonical systems are compatible with the definitions for $j$-monotonic families: 
\begin{lemma} The canonical system \eqref{canonicalsystemgeneral} is:
\begin{enumerate}[(a)]
\item in PdB-gauge if and only if the solution $\cA(z,t)$ is in PdB-gauge;
\item in A-gauge if and only if the solution $\cA(z,t)$ is in A-gauge.
\end{enumerate}
\end{lemma}

\begin{proof}
Assertion $(a)$ follows immediately by taking $z=0$ in \eqref{canonicalsystemgeneral}. To prove assertion $(b)$, observe that for $z=i$ the initial problem \eqref{canonicalsystemgeneraldifferentialform} is of the form
\[
 (\partial_\nu \cA)(i,t) j = - \cA(i,t) (\cP(t) + \cQ(t)), \qquad \cA(i,0) = I.
\]
Therefore if $\cP+\cQ$ is lower triangular, then so is $\cA(i,\cdot)$. Relations $\tr \cQ = 0$, $\tr j\cQ = 0$ are equivalent to the fact that $\cQ$ has zeroes on the diagonal. In particular, the diagonal entries of $-(\cP + \cQ)j$ are real. It follows that $\cA(i,t)$ has nonegative terms on the diagonal.

Conversely, suppose that the family $\{\cA(i,t)\}$ satisfies \eqref{A}. Since $\det\cA(i,t) = 1$ for all $t \in \bbR$, we have
$$
\cA(i,t) 
= 
\begin{pmatrix}
e^f & 0\\ * &e^{-f}
\end{pmatrix} 
$$
for some real function $f$. Then
\begin{align*}
\cP(t) +\cQ(t) &= - \cA(i,t)^{-1} (\partial_\nu \cA)(i,t) j, \\
&= 
\begin{pmatrix}
e^{-f} & 0\\ * &e^{f}
\end{pmatrix} 
\begin{pmatrix}
(\partial_{\nu}f) e^f & 0\\ * &(\partial_{\nu}f) e^{-f}
\end{pmatrix},\\ 
&=
\begin{pmatrix}
\partial_{\nu}f & 0\\ * &\partial_{\nu}f 
\end{pmatrix},
\end{align*}
is lower triangular as well. Moreover, since $\cP \ge 0$, $\tr\cP j = \tr \cQ j = 0$, and $\cQ = - \cQ^*$, we have
$$
\cP = \begin{pmatrix}
p & *\\ * &p 
\end{pmatrix},
\qquad 
\cQ = \begin{pmatrix}
iq & *\\ * &iq 
\end{pmatrix},
$$
for some real $p$, $q$. Noting that $\partial_\nu f$ is real, we conclude that $q = 0$ $\nu$-almost everywhere on $\bbR$.
\end{proof}

We now turn to the other direction, to prove that continuous $j$-monotonic families which obey the Arov normalization \eqref{A} at $z=i$ are solutions of canonical systems in Arov gauge.

\begin{proof}[Proof of Theorem \ref{thmfamilytocansystem}]
By Lemma~\ref{lemmaAgaugematrixi} applied to $\fA(i,\ell)$, we can define functions $\mu(\ell)$, $\kappa(\ell)$ so that \eqref{Agaugematrixati} holds for each $\ell$. We have
\begin{align*}\label{2aug1}
\fA(i,\ell)^{-1} 
&= \begin{pmatrix} e^{-\mu(\ell)} & 0 \\ e^{\mu(\ell)} \kappa(\ell) & e^{\mu(\ell)} \end{pmatrix},\\ 
\fA(i,\ell_2)^{-1} \fA(i,\ell_1) 
&= \begin{pmatrix} e^{\mu(\ell_1)-\mu(\ell_2)} & 0 \\ e^{\mu(\ell_1)+\mu(\ell_2)}(\kappa(\ell_2)-\kappa(\ell_1)) & e^{\mu(\ell_2)-\mu(\ell_1)} \end{pmatrix}.
\end{align*}
Lemma~\ref{lemmaAgaugematrixi} applied to $\fA(i,\ell_2)^{-1} \fA(i,\ell_1) $ for $\ell_2 \ge \ell_1$ gives $\mu(\ell_2) \ge \mu(\ell_1)$ and
\[
\lvert e^{\mu(\ell_1)+\mu(\ell_2)} ( \kappa(\ell_1) - \kappa(\ell_2)) \rvert \le e^{\mu(\ell_2)-\mu(\ell_1)} - e^{\mu(\ell_1) - \mu(\ell_2)},
\]
which can be rewritten in the form
\begin{equation}\label{eq4}
\lvert \kappa(\ell_1) - \kappa(\ell_2)\rvert \le e^{-2\mu(\ell_1)} - e^{-2\mu(\ell_2)} = 2\int_{\ell_1}^{\ell_2}e^{-2\mu(l)}\,d\mu(l).
\end{equation}
Thus, $\kappa$ is absolutely continuous with respect to $e^{-2\mu}\,d\mu$, and it can be represented in the form \eqref{kappaformula}
for some Borel measurable function $\fa$ such that $\lvert \fa \rvert \le 1$ almost everywhere with respect to $\mu$. A calculation gives
\begin{align}
\partial_\mu \fA(i,\ell)
&= 
\begin{pmatrix}
e^{\mu(\ell)} & 0 \\
-e^{\mu(\ell)}\kappa(\ell) - 2\fa(\ell) e^{-\mu(\ell)} & -e^{-\mu(\ell)}
\end{pmatrix}, \nonumber \\
- \fA(i,\ell)^{-1} \partial_\mu \fA(i,\ell) j 
&= \begin{pmatrix}  1 & 0 \\ -2\fa(\ell) & 1 \end{pmatrix} = A(\ell) + B(\ell)
 \label{2nov1}
\end{align}
with $A ,B$ given by \eqref{ABfromfa}.

Let us now show that  for any  $z \in \bbC$, $\fA(z,\ell)$ is absolutely continuous with respect to $\mu$. Schwarz lemma for $j$-contractive matrix functions \cite{P60,EP73} says that for every $z_*, z\in \bbC_+$ and every $j$-contractive $2 \times 2$ matrix function $C$ on $\bbC_+$ the following block matrix is positive semi-definite:
\begin{equation}\label{eq1}
\cC=
\begin{pmatrix}
i \frac{j - C(z_*) j C(z_*)^*}{z_* - \bar z_*} & i\frac{j-C(z_*)j C(z)^*}{z_*-\bar z} \\
i \frac{j -C(z) jC(z_*)^*}{z-\bar z_*} & i \frac{j - C(z) j C(z)^*}{z - \bar z}
\end{pmatrix}
\ge 0.
\end{equation}
Applying this to $C(z) = \fA(z,\ell_1)^{-1}\fA(z,\ell_2)$ for any $\ell_1<\ell_2$ gives
\begin{equation}\label{eq2}
\cC = \cC_{diag}(\ell_1)(\cC_{\fA}(\ell_2) - \cC_{\fA}(\ell_1))\cC_{diag}(\ell_1)^*
\end{equation}
in \eqref{eq1}, where
$$
\cC_{diag}(\ell) 
= \begin{pmatrix}
\fA(z_*, \ell)^{-1} & 0 \\
0 & \fA(z, \ell)^{-1}
\end{pmatrix},
$$
and 
$$
\cC_{\fA}(\ell) = \begin{pmatrix}
i \frac{j - \fA(z_*,\ell) j \fA(z_*,\ell)^*}{z_* - \bar z_*} & i\frac{j - \fA(z_*,\ell) j \fA(z,\ell)^*}{z_*-\bar z} \\
i\frac{j - \fA(z,\ell) j \fA(z_*,\ell)^*}{z-\bar z_*} & i \frac{j - \fA(z,\ell) j \fA(z,\ell)^*}{z - \bar z}
\end{pmatrix}.
$$
It follows from \eqref{eq1}, \eqref{eq2} that the family $\cC_{\fA}(\ell)$ is monotonic. From here we see that there exists a scalar measure $\sigma$ and $2\times2$-matrix valued mappings $C_{jk}$ such that 
\begin{equation}\label{eq5}
\cC_{\fA}(\ell) 
= \int_{[0,\ell)}\begin{pmatrix}
C_{11}& C_{12}\\
C_{21}& C_{22}\end{pmatrix}d\sigma.
\end{equation}
From now on, let $z_* = i$. Then we have 
\begin{equation}\label{eq3}
\int_{[0,\ell)} C_{11}d\sigma =
\frac{j - \fA(i,\ell) j \fA(i,\ell)^*}{2}
= \int_{[0,\ell)}\fA({i},l) A(l) \fA({i},l)^*\,d\mu(l).
\end{equation}
Let us show that $C_{12}\,d\sigma$ is absolutely continuous with respect to $\mu$. If $\mu(e) = 0$, then \eqref{eq5}, \eqref{eq3} and the monotonicity of $\cC_{\fA}$ show that
\begin{equation*}
\begin{pmatrix}
0& *\\
*& * 
\end{pmatrix}
=
\int_{e}\begin{pmatrix}
C_{11}& C_{12}\\
C_{21}& C_{22}\end{pmatrix}d\sigma \ge 0.
\end{equation*}
From here we see that $\int_{e} C_{12}\,d\sigma = \int_{e}C_{21}\,d\sigma = 0$ by considering quadratic form of the block matrix above on vectors of the form $(xe_1, e_2)$, $x \in \bbR$, $e_{1,2} \in \bbC^2$ and using the fact that any sign-definite linear function is identically constant. Thus, the mapping 
\[
\ell \mapsto \int_{[0,\ell)} C_{12}\,d\sigma = i\frac{j - \fA(i,\ell) j \fA(z,\ell)^*}{i-\bar z}
\]
is absolutely continuous with respect to $\mu$. Since we already know that $\fA(i,\cdot)$ is absolutely continuous with respect to $\mu$ and $\fA(i,\cdot) \in \SL(2, \bbC)$, from here we see that the mapping $\ell \mapsto \fA(z,\ell)$ is absolutely continuous with respect to $\mu$. It follows that we could have taken $\sigma= \mu$ from the start. In particular, by denoting
\begin{equation}\label{7nov1}
\cM(z,\ell) := -\fA(z,\ell)^{-1}\pd_\mu\fA(z,\ell)j
\end{equation}
we compute
\[
\cC_{diag}(\ell) \partial_\mu \cC_{\fA}(\ell) \cC_{diag}(\ell)^* =  \begin{pmatrix}
\frac{\cM(i,\ell) + \cM(i,\ell)^*}2 & i \frac{\cM(i,\ell) + \cM(z,\ell)^*}{i-\bar z} \\
i \frac{\cM(z,\ell) + \cM(i,\ell)^*}{z+i} & i \frac{\cM(z,\ell) + \cM(z,\ell)^*}{z-\bar z}
\end{pmatrix}  \ge 0.
\]
Positivity of the bottom right block implies that
\begin{equation}\label{7nov3}
i \frac{\cM(z,\ell) + \cM(z,\ell)^*}{z-\bar z} \ge 0
\end{equation}
and matrix positivity implies
\[
\left\lVert  i \frac{\cM(i,\ell) + \cM(z,\ell)^*}{i-\bar z}  \right\rVert^2 \le \left\lVert \frac{\cM(i,\ell) + \cM(i,\ell)^*}2  \right\rVert \left\lVert  i \frac{\cM(z,\ell) + \cM(z,\ell)^*}{z-\bar z} \right\rVert.
\]
Since $\lVert \cM(i,\ell) \rVert \le 2$ for all $\ell$, this implies an upper bound on $\lVert \cM(z,\ell) \rVert$, uniformly for $z$ in compact subsets of $\bbC_+$ and uniformly in $\ell$. 

So far, $\cM(z,\ell)$ was viewed for each $z$ as a Borel function of $\ell$, uniquely defined up to a zero measure set. Now let us consider it as a function of $z$. Fix a simple closed contour $\gamma$ in $\bbC_+$. For every $w$ in the region enclosed by $\gamma$ and every $\ell > 0$, we claim that
\begin{equation}\label{6nov1}
\fA(z,\ell) - \fA(z,0) = \int_0^\ell  \oint_\gamma \frac{\partial_\mu \fA(w,l)}{w-z} \frac{dw}{2\pi i} d\mu(l).
\end{equation}
Namely, due to uniform boundedness of the integrand, this follows by applying Fubini's theorem to exchange the integrals and using analyticity of $\fA$. Moreover, Fubini's theorem guarantees that for $\mu$-a.e. $\ell$, the contour integral is well-defined; it is by construction an analytic function of $z$ in the region enclosed by $C$. By \eqref{6nov1}, the contour integral can be taken as the new definition of $\partial_\mu \fA(z,\ell)$ and away from a zero measure set of $\ell$,  $\partial_\mu \fA(z,\ell)$ is analytic in $z$ in the region enclosed by $C$. By a countable exhaustion of $\bbC_+$, the functions $\partial_\mu \fA(z,\ell)$ are analytic in $\bbC_+$ away from a zero measure set of $\ell$. Thus, $\cM(z,\ell)$ are analytic in $\bbC_+$ away from a zero measure set of $\ell$.

By \eqref{7nov3}, $i\cM(z,\ell)$ is a matrix Herglotz function. Let us use
\[
j - \fA(z,\ell) j \fA(z,\ell)^* = \int_0^\ell \fA(z,l) ( \cM(z,l) + \cM(z,l)^*) \fA(z,l)^* d\mu(l).
\]
Evaluating at $z = \lambda + i \epsilon$ and letting $\epsilon \downarrow 0$ gives $0$ locally uniformly in $\lambda$, so by Stieltjes' inversion, the Herglotz representation of $\cM(z,\ell)$ has trivial measure on $\bbR$, and consequently it is of the form 
$i \cM(z,\ell) = A(\ell) z + i B(\ell)$ for some constant matrices $A \ge 0$ and $B= - B^*$. Evaluating $\cM(i,\ell) = A(\ell) + B(\ell)$ and $\cM(i)^* = A(\ell) - B(\ell)$ and using \eqref{2nov1}, we obtain \eqref{ABfromfa} and \eqref{canonicalsystemArov} holds for $z \in \bbC_+$. Finally, \eqref{canonicalsystemArov} holds for all $z\in \bbC$ by continuity and symmetry or by analyticity.
\end{proof}

\section{Weyl theory for $j$-monotonic families}

In this section we consider the nested family of Weyl disks \eqref{WeylDiskeqn}. We begin by characterizing the limit point case for $j$-monotonic families in the Arov gauge.

\begin{proof}[Proof of Prop.~\ref{prop18}]
A direct calculation using \eqref{Agaugematrixati} describes the Weyl disk $D(i,\ell)$ as the set of $w$ such that
$$\lvert e^{\mu(\ell)} w - e^{\mu(\ell)} \kappa(\ell) \rvert^2 \le e^{-2\mu(\ell)}.$$ 
Its Euclidean radius is $e^{-2\mu(\ell)}$ and its Euclidean center is $\kappa(\ell)$. Therefore, the disks shrink to a point if and only if $\mu(\ell) \to \infty$, and in this case,
\begin{equation}\label{eq17}
s_+(i) = \lim_{\ell\to\infty} \kappa(\ell) = \int_0^\infty 2 \fa e^{-2\mu} \,d\mu.
\end{equation}
Now consider the case where $\mu$ is a finite measure on $\bbR_+$ and assume (by possibly reparametrizing) that it is defined on a finite interval $[0,L]$. The initial value problem \eqref{canonicalsystemArov}-\eqref{ABfromfa} then has a solution $\{\fA(z,\ell)\}_{\ell \in [0,L]}$. In this case,
\[
\bigcap_{\ell \ge 0} D(z,\ell) = D(z,L) =  \{ w  \mid \begin{pmatrix} w & 1 \end{pmatrix} \fA(z,L)  j \fA(z,L)^* \begin{pmatrix} w & 1 \end{pmatrix}^* \ge 0 \}
\]
is a nontrivial disk for all $z \in \bbC_+$. Thus, if $\cap_{\ell \ge 0} D(i,\ell)$ is a disk, then $\mu$ is a finite measure and  $\cap_{\ell \ge 0} D(z,\ell)$ is a nontrivial disk for all $z \in \bbC_+$. Then this holds for all canonical systems (not only in Arov gauge), so this implication can be reversed by applying it to the transfer matrices $\tilde \cA(z,\ell) = \fA((z-x)/y, \ell)$. 
\end{proof}

Note that for every $t$ and every point $w \in \bbD$ we have $(w , 1)\cA(z,t)^{-1} \in D(z,t)$. It follows that in the limit point case we have
\begin{equation}\label{eq09}
(s_+(z), 1) \simeq \lim_{t\to+\infty}(w , 1)\cA(z,t)^{-1}, \qquad z \in \bbC_+, \quad |w|<1,
\end{equation}
where $\simeq$ stands for the projective relation in $\bbC^2 \setminus \{0\}$:
\begin{equation}\label{eq01}
\zeta_1 \simeq \zeta_2 \;\mbox{ if and only if }\; \zeta_1 = \lambda\zeta_2\; \mbox{ for some }\;\lambda \in \bbC, \qquad \zeta_{1,2} \in \bbC^2 \setminus \{0\}.
\end{equation}
This shows that in the limit point case the mapping $z \mapsto s_+(z)$ defines an analytic function $s_+$ of Schur class in $\bbC_+$ (in other words, we have $s_+(\bbC_+) \subset \overline{\bbD}$). This function is called the Schur function of the $j$-monotonic family $\{\cA(z,t)\}_{t \in \bbR}$.

Let us briefly discuss how just defined Schur functions $s_+$ are related to the Weyl functions $m_+$, the classical object of spectral theory in Potapov-de Branges gauge. Any $2\times 2$ matrix $\cJ$ with $\cJ = \cJ^* = \cJ^{-1}$ and $\cJ \neq \pm I$ is unitarily equivalent to $j$. For the choice \eqref{cJmatrixdefn}, the unitary equivalence is
 $\cJ = W j W^{-1}$ where $W=\sqrt{1/2} ( \begin{smallmatrix}  i & i \\ 1 & -1 \end{smallmatrix})$ corresponds to the Cayley transform. This makes it trivial to switch between $j$ and $\cJ$. For instance, a matrix $T$ is $j$-expanding if and only if $W T W^{-1}$ is $\cJ$-expanding. The inequality $\begin{pmatrix} w & 1 \end{pmatrix} \cJ\begin{pmatrix} w & 1 \end{pmatrix}^* > 0$ describes the upper half-plane $\bbC_+$. In fact, if we define
\begin{equation}\label{eq10}
T(z,t) = (W \cA(z,t) W^{-1})^\top
\end{equation}
for a $j$-monotonic family $\{\cA(z,t)\}_{t \in \bbR}$, then 
\begin{equation*}
T(z,t_2)^* \cJ T(z,t_2) \ge T(z,t_1)^* \cJ T(z,t_1), \qquad t_2 \ge t_1.
\end{equation*}
The property above is satisfied by transfer matrices of the canonical systems of the form
$$
\cJ T(z,t)' = -izH(t)T(z,t).
$$ 
The standard definition of Weyl disks can be expressed as
\[
D^+_T(z,t) = \{ u \mid \begin{pmatrix} u & 1 \end{pmatrix} T(z,t)^\top \cJ (T(z,t)^\top)^* \begin{pmatrix} u & 1 \end{pmatrix}^* \ge 0 \},
\]
so for $T$, $\cA$ related by \eqref{eq10} we have $u \in D^+_T(z,t)$ if and only if $w \in D_{\cA}(z,t)$ where $\begin{pmatrix} w & 1 \end{pmatrix} \simeq \begin{pmatrix} u & 1 \end{pmatrix} W$, i.e., $u$, $w$ are related by the Cayley transform: 
$$
u = i \frac{1+w}{1-w}.
$$ 
In particular, if we were working with $\cJ$, the Weyl disks would be subsets of $\overline{\bbC_+}$, and instead of the Schur function we would obtain the  function 
\[
\{m_+(z)\} = \bigcap_{t \ge 0} D_{T}(z,t),  \qquad 
m_+ : \bbC_+ \to \bbC_+,
\]
of Herglotz class in $\bbC_+$ (excluding degenerate cases when it is a constant in $\bbR \cup \{\infty\}$). This function is called the Titchmarsh--Weyl function of $\{T(z,t)\}_{t \in \bbR}$. It is related to the Schur spectral function of $\{\cA(z,t)\}_{t \in \bbR}$ by the Cayley transform:
\[
m_+(z) = i \frac{1+s_+(z)}{1-s_+(z)}, \qquad z \in \bbC_+.
\]

We must acknowledge a degenerate case: in the terminology of de Branges, the following corollary corresponds to the case when $[0,\infty)$ is a single singular interval.

\begin{corollary}\label{c01}
The function $s_+$ is a unimodular constant function if and only if $\fa$ is a unimodular constant almost everywhere on $\bbR_+$.
\end{corollary}
\begin{proof}
From \eqref{eq17} we see that $|s_+(i)| = 1$ if and only if $\fa = c$ almost everywhere on $\bbR$ for some $c \in \bbC$ such that $|c|=1$. It remains to use the classical Schwarz lemma.
\end{proof}

It is now easy to obtain a gauge-independent result about locally uniform shrinking of the Weyl disks:

\begin{proposition}
For any $j$-monotonic family in the limit point case  such that $s_+$ is not a unimodular constant and for any compact $K \subset \bbC_+$, we have
\begin{equation}\label{16aug1}
\lim_{t\to \infty} \sup_{z \in K} \diam \, D(z,t) = 0,
\end{equation}
where $\diam$ denotes the diameter in Euclidean distance.
\end{proposition}

\begin{proof}
For an arbitrary $j$-contractive matrix function $\cA(z)$
$$
 \begin{pmatrix}w&1
\end{pmatrix}(j-\cA(z_0)j\cA(z_0)^*)\begin{pmatrix}\bar w\\1
\end{pmatrix}\ge 0.
$$
On the other hand, if $w\in\bbT$ and in the Weyl disk at $z_0$, then
$$
 \begin{pmatrix}w&1
\end{pmatrix}(j-\cA(z_0)j\cA(z_0)^*)
\begin{pmatrix}\bar w\\1
\end{pmatrix}= -\begin{pmatrix}w&1
\end{pmatrix}\cA(z_0)j\cA(z_0)^*\begin{pmatrix}\bar w\\1
\end{pmatrix}\le 0.
$$
Thus 
$$
 \begin{pmatrix}w&1
\end{pmatrix}(j-\cA(z_0)j\cA(z_0)^*)\begin{pmatrix}\bar w\\1
\end{pmatrix}= 0.
$$
and by the Schwarz Lemma $
 \begin{pmatrix}w&1
\end{pmatrix}\cA(z_0)=\begin{pmatrix} w &1
\end{pmatrix}\cA(z)
$ for all $z\in\bbC_+$.

If $s_+$ is not a unimodular constant, there exists $\ell_1 > 0$ such that $D(i,\ell_1) \subset \bbD$; the above discussion, applied to $\cA(z,\ell_1)$, shows by contraposition that $D(z,\ell_1) \subset \bbD$ for all $z\in \bbC_+$. Thus, $z \mapsto \diam D(z,\ell)$ for $\ell \ge \ell_1$ are continuous functions on $\bbC_+$. They converge to zero monotonically on $\bbC_+$, so \eqref{16aug1} follows from Dini's theorem.
\end{proof}

\begin{corollary}\label{l03}
	Let $T=\left(\begin{smallmatrix}T_{11} & T_{12}\\ T_{21} & T_{22}\end{smallmatrix}\right)$ be a $j$-contractive matrix, and let $D$ be its Weyl disk:
	$$
	D = \{w: \; (w , 1)TjT^*(w , 1)^* \ge 0\}.
	$$
	Then $\diam\,D = 2/(|T_{11}|^2-|T_{12}|^2)$.
\end{corollary}
\begin{proof}
	We have 
	$$
	|T_{11}|^2-|T_{12}|^2 = -(1, 0) T j T^*(1, 0)^*.
	$$
	In particular, this quantity is invariant under the transformation $T \mapsto TU$ for every $j$-unitary matrix $U$. Using Proposition \ref{l07}, choose $U$ so that $TU$ is lower-triangular. Then $T_{12} = 0$ and the proof of Proposition \ref{prop18} shows that the diameter is indeed $2/|T_{11}|^2$ (that is equal, in notation of Proposition \ref{prop18}, to $2e^{-2\mu}$). Since the Weyl disk is invariant under the transformation $T \mapsto TU$, $U \in \SU(1,1)$, the lemma follows. 
\end{proof}

We next translate the famous de Branges' uniqueness theorem into A-gauge:

\begin{proof}[Proof of  Theorem~\ref{theoremdeBrangesuniqueness}]
If the family $\fA(z,\ell)$ in A-gauge has the Schur function $s_+$, passing to PdB-gauge we obtain a family in PdB-gauge $\cA(z,\ell) = \fA(z,\ell) \fA(0,\ell)^{-1}$ with the same Schur function. By de Branges' uniqueness theorem \cite{dB}, that describes the family $\cA$ uniquely up to reparametrization; thus, the family $\fA$ is also determined uniquely up to reparametrization by Prop.~\ref{l07}. Thus, the parameters in A-gauge are determined uniquely up to reparametrization.
\end{proof}

We now consider the reflection procedure \eqref{canonicalsystemreflectiongeneral}. 
Using the relation $j_1 j j_1 = - j$, we see that the family $\tilde \cA(z,t) = j_1 \cA(z,-t) j_1$ is indeed $j$-monotonic. We say that the initial family $\{\cA(z,t)\}_{t\le 0}$ is in the limit point case at $-\infty$ if the reflected family $\{ \tilde \cA(z,t) \}_{t \ge 0}$ is in the limit point case at $+\infty$. In the latter case, we denote this Schur function by $s_-$ and characterize it by applying \eqref{eq09} to $\tilde \cA(z,t)$, which gives:
\begin{equation}\label{sminusaslimit}
\begin{pmatrix}s_-(z), 1\end{pmatrix} \simeq \lim_{t \to +\infty} \begin{pmatrix} w & 1 \end{pmatrix}j_1\cA(z, -t)^{-1}j_1, \qquad z \in \bbC_+.
\end{equation}
In this way we associate two Schur functions $s_\pm$ to each two-sided $j$-monotonic family $\{\cA(z,t)\}_{t \in \bbR}$.
This is consistent with the literature on canonical systems in Potapov-de Branges' gauge, in which reflection 
corresponds to the change 
$\tilde T(z,x) = j T(z,-t) j$, and the $\cJ$-monotonic transfer matrices $T(z,x)$ are related to the $j$-monotonic transfer matrices $\cA(z,t)$ as in  \eqref{eq10}. For instance, the definition of the negative half-line Weyl function $m_-$  corresponds to our negative half-line Schur function by $m_- = i \frac{1+s_-}{1-s_-}$.

As noted in the introduction, this reflection procedure does not preserve A-gauge. Prop.~\ref{lemmanegativehalflineAgauge} describes a modified reflection procedure which does: 

\begin{proof}[Proof of Prop.~\ref{lemmanegativehalflineAgauge}]
A direct calculation shows
	\[
	\partial_{\overleftarrow{\mu}} (j_1 \fA(z,-\ell) j_1) = - j_1 \fA(z,-\ell) j_1 j_1 (-iz A(-\ell) + B(-\ell)) j_1
	\]
	and
	\[
	j_1 (-izA(-\ell) + B(-\ell)) j_1 = \begin{pmatrix}-iz&-\fa(-\ell)(-iz+1)\\
	-\overline{\fa(-\ell)}(-iz-1)&-iz
	\end{pmatrix}.
	\]
	Note that this is upper-triangular at $z=i$ instead of lower-triangular, and
	\begin{multline*}
	\begin{pmatrix} v(z)&0\\0&1
	\end{pmatrix}
	\begin{pmatrix}-iz&-\fa(-\ell)(-iz+1)\\
	-\overline{\fa(-\ell)}(-iz-1)&-iz
	\end{pmatrix}
	\begin{pmatrix} v(z)&0\\0&1
	\end{pmatrix}^{-1}=
	\\
	=\begin{pmatrix}-iz&-\fa(-\ell)(-iz-1)\\
	-\overline{\fa(-\ell)}(-iz+1)&-iz
	\end{pmatrix}
	\end{multline*}
	so $\overleftarrow \fA(z,\ell)$ defined by \eqref{3jun1} is a canonical system in the Arov gauge with the coefficients $(\overleftarrow \fa, \overleftarrow{\mu})$. Comparing the Weyl disks of $\tilde \fA(z,\ell)$ and $j_1 \fA(z,-\ell)j_1$  gives $s_-(z) = v(z) \overleftarrow{s_+}(z)$.
\end{proof}

\section{Ricatti equation and spectral asymptotics}

We now study the family of Schur functions generated by the coefficient stripping
\[
\fA^{[\ell]}(z,l) = \fA(z,\ell)^{-1} \fA(z,l + \ell).
\]
Just as the Schur functions $s_\pm$ can be characterized by \eqref{eq09}, \eqref{sminusaslimit}, the family of Schur functions $s_\pm(z,\ell)$ obey

\begin{align}\label{eq06}
&\begin{pmatrix}s_+(z,\ell), 1\end{pmatrix} \simeq \lim_{l \to +\infty} \begin{pmatrix} w & 1 \end{pmatrix}\fA^{[\ell]}(z, l)^{-1}
\simeq \begin{pmatrix}s_+(z), 1\end{pmatrix} \fA(z,\ell).\\
&\begin{pmatrix}s_-(z,\ell), 1\end{pmatrix} \simeq \lim_{l \to +\infty} \begin{pmatrix} w & 1 \end{pmatrix}(j_1\fA^{[\ell]}(z, -l)j_1)^{-1}
\simeq \begin{pmatrix}s_-(z), 1\end{pmatrix} j_1\fA(z, \ell)j_1.\label{eq05}
\end{align}
Moreover, for every compact $K \subset \bbC_+$, the convergence is uniform with respect to $z \in K$ and $w \in \overline{\bbD}$. Next result shows that functions $s_+(z,\cdot)$ satisfy a Ricatti equation for every $z \in \bbC_+$.

\begin{proof}[Proof of Prop.~\ref{propRicatti}]
Formula \eqref{eq06} implies the existence of a non-vanishing function $\phi(z,\cdot)$ such that
\begin{equation}\label{10jun1}
\begin{pmatrix} s_+(z) & 1 \end{pmatrix} \fA(z,\ell) = \phi(z,\ell) \begin{pmatrix} s_+(z,\ell) & 1 \end{pmatrix}, \qquad z \in \bbC, \quad \ell \in \bbR.
\end{equation}
Since $\fA(z, \cdot)$ is absolutely continuous with respect to $\mu$, the same is true for
$$
\phi(z,\ell) = \begin{pmatrix} s_+(z) & 1 \end{pmatrix} \fA(z,\ell) \begin{pmatrix} 0 & 1 \end{pmatrix}^*.
$$
Multiplying \eqref{10jun1} by $j$ from the right and using $\partial_\mu \fA j = \fA (izA-B)$, we obtain
\[
\begin{pmatrix} s_+(z) & 1 \end{pmatrix} \fA(z,\ell) (izA-B) = \partial_\mu \phi(z,\ell) \begin{pmatrix} s_+(z,\ell) & 1 \end{pmatrix}j +  \phi(z,\ell) \begin{pmatrix} \partial_\mu s_+(z,\ell) & 0 \end{pmatrix}j,
\]
or, equivalently,
\begin{multline}\label{eq08}
\phi(z,\ell) \begin{pmatrix} s_+(z,t) & 1 \end{pmatrix} (izA-B) =\\ = \partial_\mu \phi(z,\ell) \begin{pmatrix} s_+(z,\ell) & 1 \end{pmatrix}j +  \phi(z,\ell) \begin{pmatrix} \partial_\mu s_+(z,\ell) & 0 \end{pmatrix}j.
\end{multline}
Now right-multiplying the equation by $ \binom 1{s_+(z,\ell)}$ concludes the proof.
\end{proof}

An arbitrary Schur function has nontangential limits Lebesgue-a.e. on $\bbR\cup\{\infty\}$, but is not guaranteed to have a nontangential (or normal) limit at $\infty$. Since a canonical system can have an arbitrary Schur function, there can be no general statements about the behavior of the Schur function at $\infty$. In this sense the situation is very different compared to the Schr\"odinger or Dirac cases, in which leading order asymptotic behavior is obtained by comparing to a certain ``free" operator. However, we prove a Lebesgue-type condition suffices:

\begin{theorem} \label{thmSchurlimit} 
For any canonical system in  A-gauge such that
\begin{equation}\label{eq07}
\lim_{\ell \downarrow 0} \frac 1{ \mu([0,\ell)) } \int_0^\ell  \lvert \fa(l) - \fa(0) \rvert  d\mu(l) = 0,
\end{equation}
for some $\fa(0) \in \bbC$, then the Schur function has a nontangential limit at $+i\infty$,
\begin{equation}\label{8julg}
\lim_{\substack{ z\to \infty \\ \arg z \in [\delta,\pi-\delta]}} s_+(z) = \frac{ \fa(0) }{ 1+ \sqrt{ 1-  \lvert \fa(0) \rvert^2}},
\end{equation}
for any $\delta > 0$.
\end{theorem}
We note that for PdB gauge the analogous statement was proved in Theorem 3.1 in \cite{EKT18} by a different method.

The proof of Theorem \ref{thmSchurlimit} relies on a $z$-dependent rescaling of the Ricatti equation, following an idea used for Dirac operators by Clark--Gesztesy \cite{CG}. To prepare for this, we first reparametrize the $j$-monotonic family. 

\begin{remark}[reparametrizing to Lebesgue measure] \label{remarkLebesgue}
Consider the $j$-contractive family defined by
\[
\tilde \fA(z,\mu(\ell)) = \fA(z, \ell).
\]
Note that this is well-defined even if $\mu$ is not injective, since $\mu$ is a continuous increasing function and on any intervals on which $\mu$ is constant, $\fA$ is constant as well. Comparing with \eqref{Agaugematrixati} shows that $\tilde\fA$ has A-gauge parameters $\tilde \mu(\mu(\ell)) = \mu(\ell)$, $\tilde \fa(\mu(\ell)) = \fa(\ell)$. In particular, $\tilde \fA$ has Lebesgue measure as its parameter. Conditions and conclusions such as those in Theorem~\ref{thmSchurlimit} are invariant with respect to this reparametrization; in other words, we can assume throughout the proof that $\mu$ is Lebesgue measure.
\end{remark}
The Ricatti equation \eqref{Ricattiequation} takes the form
\[
\partial_\ell s_+(z,\ell) = -\begin{pmatrix} s_+(z,\ell) & 1 \end{pmatrix} (iz A(\ell) -B(\ell))  \begin{pmatrix} 1 \\ s_+(z,\ell) \end{pmatrix}.
\]
We introduce a $z$-dependent ``fast variable'' $r = \lvert z \rvert \ell$ and define
\[
s(z,r) = s_{+}\left(z, r/|z| \right),
\]
which obeys the differential equation
\begin{equation}\label{8julc}
\partial_r s(z, r) =  -\begin{pmatrix} s(z,r) & 1 \end{pmatrix} \left(\frac{ iz}{|z|} A\left(\frac{r}{|z|}\right) -\frac{1}{|z|} B \left(\frac{r}{|z|}\right) \right)  \begin{pmatrix} 1 \\ s(z,r) \end{pmatrix}.
\end{equation}
We are going to take a limit in \eqref{8julc} when $z$ nontangentially tends to $+i\infty$. For this we need the following lemma.
\begin{lemma}\label{l1}
Assume that $\mu$ is Lebesgue measure and \eqref{eq07} holds. Define $A(0)$ by \eqref{ABfromfa} using the value $\fa(0)$. Let a sequence $\{z_k\}$ be such that $\lvert z_k \rvert \to \infty$, $\arg z_k \to \phi + \pi/2$ for some $\phi \in (-\pi/2,\pi/2)$, and $s(z_k, 0) \to s_0$ for some $s_0 \in \overline{\bbD}$. Then the initial value problem
\begin{equation}\label{7jula}
\eta'(r) = e^{i\phi} \begin{pmatrix} \eta(r) & 1 \end{pmatrix} \cP(0)  \begin{pmatrix} 1 \\ \eta(r) \end{pmatrix}, \qquad \eta(0) = s_0
\end{equation}
has a global solution which obeys $\eta(r) \in \overline{\bbD}$ for all $r \in [0,\infty)$, and the functions $s(z_k, \cdot)$ converge uniformly on compacts to $\eta$ as $k\to\infty$.
\end{lemma}

\begin{proof}
Since $|\fa(0)| \le 1$, there exists $c > 0$ independent of $s_0\in\overline{\bbD}$ such that the initial value problem \eqref{7jula} has a solution on $[0,c]$ such that $\lvert \eta \rvert \le 2$. Define $L \in (0, +\infty]$ to be the supremum of all such $c$. Later on, we will work with $r \in [0,L)$ and $\lvert z \rvert \ge 1$.  Note that for such $r$, $z$ we have
\[
 \left\lVert\cR(z,r) \right\rVert \le 3, \qquad \cR(z,r)  = -\frac{ iz}{|z|} A\left(\frac{r}{|z|}\right) +\frac{1}{|z|} B \left(\frac{r}{|z|}\right).  
\]
We also have $\|A(0)\| \le 2$ by construction. Using the identity
$$
(T_1 h_1, h_1) - (T_2 h_2, h_2) = 
((T_1 - T_2) h_1, h_1) + (T_2 (h_1 - h_2), h_2) + (T_2 h_1, h_1 - h_2)
$$
for linear operators, we obtain
from \eqref{8julc}, \eqref{7jula} the estimate 
\begin{equation}\label{11jula}
|f'(z,r)|
\le 4 \left\lVert \cR(z,r) - e^{i\phi} A(0) \right\rVert + 10|f(z,r)|, \quad f(z,r) = s(z,r) - \eta(r),
\end{equation}
on $[0,L)\times \{|z|\ge 1\}$.
Since $f$ is absolutely continuous with respect to $r$, the same is true for the function $h(z,r) = e^{-20 r} \lvert f(z,r) \rvert^2$.  Moreover, we have
\[
h'(r) = 2 e^{-20 r} \Re\left( \overline{f(r)} f'(r) \right) - 20 e^{-20 r} \lvert f(r) \rvert^2,
\]
hence
\begin{align*}
h'(r) 
&\le 2 e^{-20 r}|f(r)| \left(4 \left\lVert \cR(z,r) - e^{i\phi} A(0) \right\rVert + 10|f(z,r)|\right)- 20 e^{-20 r} \lvert f(r) \rvert^2 \\
&\le 8 e^{-20 r}|f(r)| \left\lVert \cR(z,r) - e^{i\phi} \cP(0) \right\rVert \le 24 \left\lVert \cR(z,r) - e^{i\phi} A(0) \right\rVert.
\end{align*}
Integrating last inequality, we obtain
\begin{equation}\label{8jule}
\lvert f(z,r)\rvert^2 \le e^{20 r}\left( \lvert f(z,0) \rvert^2 + 24 \int_0^r \left\lVert \cR(z,t) - e^{i\phi}  A(0) \right\rVert dt \right).
\end{equation}
Since $0$ is a Lebesgue point for $\fa$, we have
\[
\int_0^r \left\lVert \cR(z_k, t) - e^{i\phi} A(0) \right\rVert dt \to 0
\]
along any sequence $z_k$ with $\lvert z_k \rvert \to \infty$ and $\arg(-i z_k) \to \phi$. Thus, \eqref{8jule} together with our assumption $s(z_k, 0) \to s_0 = \eta(0)$ implies that $s(z_k, r) \to \eta(r)$ uniformly in $r$ on compact subsets of $[0,L)$. It remains to show that $L = +\infty$. To this end, note that if this is not the case, we have $\lvert \eta(r_0) \rvert = 3/2$ for some $r_0 \in [0, L)$. But the above argument shows that $\eta(r_0) = \lim_{k}s(z_k, r_0)$ is in the unit disk, thus giving a contradiction.
\end{proof}
It turns out that condition $\eta(r) \in \bbD$, $r \ge 0$, for a solution $\eta$ of \eqref{7jula} determines $\eta$ uniquely.
\begin{lemma}\label{l2}
Let $a \in \overline{\bbD}$ and $\phi \in (-\pi/2, \pi/2)$. If $\eta : [0,\infty) \to \overline{\bbD}$ solves the differential equation
\begin{equation}\label{7julb}
\eta'(r) = e^{i\phi} \begin{pmatrix} \eta(r) & 1 \end{pmatrix}  \begin{pmatrix} 1 & - \overline{a} \\ -a  & 1 \end{pmatrix} \begin{pmatrix} 1 \\ \eta(r) \end{pmatrix}
\end{equation}
then $\eta(0) =  \frac{ a }{ 1+ \sqrt{ 1-  \lvert a\rvert^2}}$. In other words, for any other initial value $\eta(0)$, the solution of the initial value problem exits $\overline{\bbD}$ in finite time. \end{lemma}

\begin{proof}
If $ a = 0$, then the Ricatti equation reduces to $\eta'(r) = 2 e^{i\phi} \eta(r)$, so the general solution is $\eta(r) = e^{2 r e^{i\phi}} \eta(0)$. In particular, any initial value $\eta(0) \neq 0$ gives an exponentially growing solution.

\medskip

If $a \neq 0$, the general solution can be found explicitly as $\eta = w_1 /w_2$ where $w = \begin{pmatrix}w_1 & w_2 \end{pmatrix}$ solves the linear equation
\begin{equation}\label{8jula}
w' = - e^{i\phi} w \begin{pmatrix} 1 & - \overline{a} \\ -a  & 1 \end{pmatrix} j.
\end{equation}
To see this, let us multiply \eqref{8jula} by $\begin{pmatrix} 1 & - w_1/w_2 \end{pmatrix}^\top$ from the right,
$$
w'_1 - \frac{w_1}{w_2} w'_2 
= e^{i\phi}w_2\begin{pmatrix}w_1/w_2 & 1 \end{pmatrix} \begin{pmatrix} 1 & - \overline{a} \\ -a  & 1 \end{pmatrix}\begin{pmatrix} 1 \\  \frac{w_1}{w_2} \end{pmatrix},
$$
and note that this equation is equivalent to \eqref{7julb} if $\eta = w_1/w_2$. Put $\rho = \sqrt{1-|a|^2}$. Observe that 
for the matrix
$$
T = - \begin{pmatrix} 1 & - \overline{a} \\ -a  & 1 \end{pmatrix} j = 
 \begin{pmatrix} 1 &  \overline{a} \\ -a  & -1 \end{pmatrix}
$$
we have
$$
\begin{pmatrix}a& 1\mp\rho\end{pmatrix} T  = \pm\rho\begin{pmatrix}a & 1\mp\rho\end{pmatrix},
$$
so that if $0 < \lvert a\rvert < 1$, the general solution of \eqref{8jula} has the form 
$$
\begin{pmatrix}
w_1 & w_2
\end{pmatrix}
= c_1 e^{\rho r e^{i\phi}} \begin{pmatrix} a & 1 -\rho \end{pmatrix} + c_2 e^{-\rho r e^{i\phi}} \begin{pmatrix} a & 1+\rho \end{pmatrix}
$$
for some constants $c_1$, $c_2$. Since  $\lvert a\rvert > 1 - \rho$, we must have $c_1 = 0$ if the condition  $|\eta(r)| \le 1$ holds for large $r \ge 0$. Meanwhile, if $c_1 = 0$ and $c_2 \neq 0$, then $\eta$ is the constant function $\frac a{1+\rho}$ and the result follows.
Now consider the case where $\lvert a\rvert = 1$. Then the general solution of \eqref{8jula} has the form 
\begin{align*}
\begin{pmatrix}
w_1 & w_2
\end{pmatrix} 
&= c_1 \begin{pmatrix} a & 1 \end{pmatrix} + c_2 \left(r\begin{pmatrix} a & 1 \end{pmatrix} + e^{-i\phi}\begin{pmatrix} 0 & -1  \end{pmatrix}\right), \\
&= \begin{pmatrix} a(c_1 + c_2r) & c_1 + c_2 -e^{-i\phi} \end{pmatrix}.
\end{align*}
We see that 
$$
\eta(r) =  \frac{a(c_1 + c_2r)}{c_1 + c_2 -c_2e^{-i\phi}}
$$
is in $\bbD$ for large $r$ if and only if $c_2 = 0$, in which case $\eta$ is again the constant function, $\eta(0) = \eta(r) = a$.
\end{proof}

\begin{proof}[Proof of Theorem~\ref{thmSchurlimit}]
Since $s_{+}$ takes values in $\overline{\bbD}$, by Lemma \ref{l1} and Lemma \ref{l2} it suffices to show that for fixed $\delta > 0$, the values of $s_+(z)$ have only one accumulation point as $\lvert z \rvert \to \infty$, $\arg z \in [\delta,\pi-\delta]$. For this, assume that $s_+(z_k)$ is convergent for some sequence $z_k\to \infty$ with $\arg z_k \in [\delta,\pi-\delta]$. By compactness, we can pass to a subsequence such that $\arg (-iz_k)$ converges. Then the limit of $s(z_k)$ is equal to $\frac{\fa(0)}{1+ \sqrt{1 - \lvert \fa(0)\rvert^2}}$. Thus, this is the only accumulation point, so \eqref{8julg} holds.
\end{proof}

\begin{proof}[Proof of Prop.~\ref{cor218}]
This follows from Theorem~\ref{thmSchurlimit} by the well-known fact that almost every point of a locally integrable function with respect to $\mu$ is a Lebesgue point of this function with respect to $\mu$.
\end{proof}

\section{Krein-de Branges formula for exponential type}\label{s5} 
As a corollary of  Theorem \ref{thmSchurlimit}, in this section we give a new proof of the de Branges mean type theorem (see Section 39 in \cite{dB}) discovered by Krein \cite{Krein97}. 

\begin{proof}[Proof of Theorem~\ref{thjul10}]
Using Prop.~\ref{l07}, let us pass to  the $j$-monotonic family in A-gauge $\fA(z,t) = \cA(z,t) U(t)$ where $U(t) \in \SU(1,1)$. The coefficients of the new canonical system obey  
\[
\hat\nu = \nu, \qquad \hat \cP = U^{-1} \cP jUj, \qquad \hat \cQ = U^{-1}  \cQ jUj - U^{-1} \partial_\nu U j.
\]
Since 
$$
\limsup_{z\to\infty}\frac{\log\| \cA(z,t)U(t)\|}{|z|} = \limsup_{z\to\infty}\frac{\log\| \cA(z,t)\|}{|z|} 
$$
and $\det \hat \cP(t) = \det \cP(t)$, it suffices to prove the theorem for $j$-monotonic families in A-gauge. 

 Formula \eqref{eq08} in the proof of Proposition \ref{propRicatti} gives
\begin{align*}
\partial_\mu \phi(z,\ell) 
&= \phi(z,\ell) \begin{pmatrix} s_+(z,\ell) & 1 \end{pmatrix} \begin{pmatrix}(-iz-1)\overline{\fa(\ell)}  \\iz  \end{pmatrix}\\
&=-\phi(z,\ell)((iz+1)s_+(z,\ell)\overline{\fa(\ell)}-iz).
\end{align*}
Since $\phi(z,0) = 1$, integrating and evaluating at $z = iy$ we get
\begin{equation}\label{112jun111}
\log\frac{1}{\phi(iy,\ell)}=\int_0^\ell \left((-y+1)s_+(iy,l)\overline{\fa(l)}+y \right) \, d\mu(l).
\end{equation}
Let us divide both sides by $y$ and pass to the limits as $y\to+\infty$. According to Prop.~\ref{cor218}, the right hand side will tend to the desired limit
\[
\int_{0}^{\ell}\left(- \frac{\lvert \fa\rvert^2}{1 + \sqrt{1- \lvert \fa\rvert^2}} + 1\right) \,d\mu = \int_{0}^{\ell}\sqrt{1- \lvert \fa\rvert^2}\,d\mu =
\int_{0}^{\ell}\sqrt{\det A}\,d\mu. 
\]
On the other hand, from \eqref{10jun1} we see that
\begin{align*}
\frac{1}{\phi(z,\ell)} 
&= \begin{pmatrix}
s_+(z,\ell) & 1 
\end{pmatrix}
\cA^{-1}(z,\ell)\begin{pmatrix}
0 \\ 1 
\end{pmatrix} \\
&= \begin{pmatrix}
s_+(z,\ell) & 1 
\end{pmatrix}
\begin{pmatrix}a_{22}&-a_{12}\\
-a_{21}&a_{11}
\end{pmatrix}\begin{pmatrix}
0 \\ 1 
\end{pmatrix} \\
&=a_{11}(z,\ell)-a_{12}(z,\ell)s_+(z,\ell)\\
&=a_{11}(z,\ell)\left(1-\frac{a_{12}(z,\ell)}{a_{11}(z,\ell)} s_+(z,\ell) \right).
\end{align*}
Since the family $\cA(z,\ell)$ is $j$-monotonic, we have $|a_{11}(z,\ell)|^2-|a_{12}(z,\ell)|^2\ge 1$ for all $\ell \ge 0$ and  $z \in \bbC_+$. It follows that the analytic function $1-\frac{a_{12}(\cdot,\ell)}{a_{11}(\cdot,\ell)}s_+(\cdot,\ell)$ has a positive real part in $\bbC_+$ and therefore, it is outer.
Since $a_{11}(z,\ell)$ is an entire function of bounded characteristic in the upper half plane without zeros, we have
$$
\lim_{y\to\infty}\frac{\log |a_{11}(iy,\ell)|}{y}=\lim_{y\to\infty} \frac 1{y}\log \frac{1}{|\phi(iy,\ell)|}=\int_{0}^{\ell}\sqrt{\det A}\,d\mu. 
$$
Since $\cA(z,\ell)\in\SL(2,\bbC)$ is $j$-contractive, its entries obey in the upper half plane
$$
1+|a_{12}(z,\ell)|^2\le |a_{11}(z,\ell)|^2,\quad  |a_{22}(z,\ell)|^2\le 1+|a_{21}(z,\ell)|^2\le |a_{11}(z,\ell)|^2.
$$
Thus, by the symmetry $\|\cA(z,\ell)\|=\|\cA(\bar z,\ell)\|$,
\begin{equation}\label{31oct20}
 \lim_{y\to\pm \infty} \frac{\log\|\cA(iy,\ell)\| }{y}=\int_{0}^{\ell}\sqrt{\det A}\,d\mu. 
\end{equation}
Finally, since the entries of $\cA(z,\ell)$ are entire functions of bounded characteristic in upper/lower half plane, the normal limit \eqref{31oct20} implies the unrestricted limit as $z \to\infty$ \eqref{9jul420}--\eqref{9jul4} \cite{Kre,Krein97,BorSod} and \cite[Chapter I, \S 10]{dB}.
\end{proof}

\section{Breimesser-Pearson theorem for $j$-monotonic families}\label{s6}
In this section we prove a version of the Breimesser-Pearson theorem \cite{BP03}. After the seminal work of Remling \cite{Rem11}, Breimesser-Pearson theorem became the standard tool in the study of reflectionless and almost periodic operators. For canonical systems in Potapov-de Brange's gauge it was first proved by Acharya \cite{Ach16}. In principle, results of this section can be obtained from an extended version of Acharya's theorem via the ``twisted shifts'' technique used in Section~7 of \cite{Remling}. However, we prefer to give a direct proof here to make the paper more self-contained. We discuss only the ``spectral part'' and give references to results in function theory appearing in the proof.

\medskip
 
Given a $j$-monotonic family $\{\cA(z,\ell)\}_{\ell \in \bbR}$, recall that its Schur spectral functions~$s_\pm$
are given by 
\begin{align}\label{eq12}
&\begin{pmatrix}s_+(z), 1\end{pmatrix} \simeq \lim_{\tau \to +\infty} \begin{pmatrix} w & 1 \end{pmatrix}\cA(z, \tau)^{-1},\\
&\begin{pmatrix}s_-(z), 1\end{pmatrix} \simeq \lim_{\tau \to +\infty} \begin{pmatrix} w & 1 \end{pmatrix}(j_1\cA(z, -\tau)j_1)^{-1},
\label{eq13}
\end{align}
(these limits do not dependent on $w \in \bbD$) and their shifted versions by
\begin{align}\label{eq15}
&\begin{pmatrix}s_+(z,\ell), 1\end{pmatrix} \simeq \begin{pmatrix}s_+(z), 1\end{pmatrix} \cA(z,\ell),\\
&\begin{pmatrix}s_-(z,\ell), 1\end{pmatrix} \simeq \begin{pmatrix}s_-(z), 1\end{pmatrix} j_1\cA(z,\ell)j_1.\label{eq16}
\end{align}
Since $s_{\pm}$ are Schur functions, for Lebesgue almost every $x \in \bbR$ there exist the limits 
\begin{equation}\label{eq11}
s_{\pm}(x) = \lim_{\e\to+0}s_{\pm}(x+i\e).
\end{equation}
Define the absolutely continuous spectrum $\sE$ of the $j$-monotonic family $\{\cA(z,\ell)\}_{\ell \in \bbR}$ to be the essential closure of the set $\left\{x \in \bbR: \; |s_{+}(x)| + |s_{-}(x)| < 2\right\}$ with respect to the Lebesgue measure. For $z \in \bbD$, let $\omega_z$ be the harmonic measure in $\bbD$: 
$$
\omega_{z}(t) = \int_{S}\frac{1-|z|^2}{|1 - \bar \xi z|^2}\,dm(\xi), \qquad S \subset \bbT,
$$
where $m$ denotes the normalized Lebesgue measure on the unit circle $\bbT$, $S$ is a measurable subset of $\bbT$. 

\medskip

Here is our version of the Breimesser-Pearson theorem. 
\begin{theorem}\label{BPthm}
	Let $\{\cA(z,\ell)\}_{\ell \in \bbR}$ be a $j$-monotonic family in the limit point case at $\pm\infty$, and let $\sE$ be its absolutely continuous spectrum. Assume that $|s_\pm(i,\ell)| < 1$ for every $\ell \ge 0$. Then for every Borel subsets $e \subset \sE$, $|e| < \infty$, $S \subset \bbT$,   we have
	$$
	\lim_{\ell \to +\infty}\left(\int_{e}\omega_{s_-(x,\ell)}(\bar S)  - 
	\int_{e}\omega_{s_+(x,\ell)}(t)\,dx\right) = 0.
	$$
\end{theorem}
\begin{proof}
Define the function
\[
\gamma(w,z) = \frac{ 2 \lvert w - z \rvert}{\sqrt{ 1-  \lvert w \rvert^2} \sqrt{ 1-  \lvert z \rvert^2}}, \qquad w,z \in \bbD, 
\]
which is related to the hyperbolic distance on $\bbD$. Using the Cayley transform and the estimate (7.21) in \cite{Remling} it is easy to check that
\begin{equation}\label{eq14}
\lvert \omega_w(t) - \omega_z(t) \rvert \le \gamma(w,z), \qquad w,z \in \bbD, \qquad S \subset \bbT.
\end{equation}
It is also possible to check (directly or using the Cayley transform and Proposition 3 in \cite{BP03}) that $\gamma$ is invariant under the M\"obius transformations.
Let us fix a finite partition $e = \cup_{0}^{N}e_j$ and a collection $\{s_j\}_{j = 1}^{N}$ such that $|e_0| < \e$, for $j \ge 1$ the sets $e_j$ are bounded, and 
	$$
	\gamma(s_+(x), s_j)<\e, \qquad x \in e_j.
	$$ 
It is possible to choose $y_*>0$ so small that for every function $s$ of Schur class in $\bbC_+$, for every $0<y<y_*$ and every $1 \le j \le N$ we have
\begin{equation}\label{eq02}
\sup_{S \subset \bbT}\left|\int_{e_j}\omega_{s(x+iy)}(t) - \int_{e_j}\omega_{s(x)}(t)\right| \le \e|e_j|.
\end{equation}
This is the key analytic ingredient of the Breimesser-Pearson theorem, and we refer the reader to Theorem 1 in \cite{BP03} or to Theorem 7.6 in \cite{Remling} for its proof. One need to use the conformal invariance of the harmonic measure to get \eqref{eq02} from those results. 
Fix $1 \le j \le N$ and consider the function $s$ defined by
$$
\begin{pmatrix} s(z,\ell) , 1\end{pmatrix} \simeq \begin{pmatrix} \overline{s_j}, 1\end{pmatrix} \overline{\cA(\bar z,\ell)}, \qquad z \in \bbC_+, \quad \ell \ge 0.
$$ 
For $x \in e_j$, we have $x = \bar x$ and
\begin{align*}
\gamma\left(s^+(x,\ell), \overline{s(x,\ell)}\right) 
&= \gamma\left(\begin{pmatrix}s_+(x), 1\end{pmatrix} \cA(x,\ell), \begin{pmatrix}s_j, 1\end{pmatrix} \cA(x,\ell)\right) \\
&= \gamma(\begin{pmatrix}s_+(x), 1\end{pmatrix} , \begin{pmatrix}s_j, 1\end{pmatrix}) \le \e,
\end{align*}
due to invariance of $\gamma$ with respect to M\"obius transformations (in the formula above we identified complex numbers with elements of the projective complex plane so that $z \in \bbC$ corresponds to $(z, 1)$). Using \eqref{eq14} we see that
$$
\left|\omega_{s^+(x,\ell)}(S) - \omega_{\overline{s(x,\ell)}}(S)\right| \le
\gamma\left(s^+(x,\ell), \overline{s(x,\ell)}\right) \le \e, \qquad x \in e_j.
$$
Integrating this inequality, we get
$$
\left|\int_{e_j}\omega_{s^+(x,\ell)}(S)\,dx - \int_{e}\omega_{\overline{s(x,\ell)}}(S)\,dx\right| \le \e|e_j|. 
$$
Let us show that for large $\ell>0$ we have
$$
\left|\int_{e_j}\omega_{s^-(x,\ell)}(\overline{S}) - \int_{e}\omega_{\overline{s(x,\ell)}}(S)\right| \le 3\e|e_j|,
$$
then the statement will follow. By definition, we have $\omega_{w}(\overline{S}) = \omega_{\bar w}(S)$. So, we need to show that
$$
\left|\int_{e_j}\omega_{s^-(x,\ell)}(\overline{S}) - \int_{e_j}\omega_{s(x,\ell)}(\overline{S})\right| \le 3\e|e_j|.
$$
Relation \eqref{eq02} reduces this to the inequality
\begin{equation}\label{eq006}
\left|\int_{e_j}\omega_{s^{-}(x+iy_*, \ell)}(\overline{S}) - \int_{e_j}\omega_{s(x+iy_*, \ell)}(\overline{S})\right| \le \e|e_j|.
\end{equation}
Estimate \eqref{eq14} and our definition of $s$ imply that it suffices to check that 
\begin{equation}\label{eq03}
\lim_{\ell \to +\infty} \gamma((s^{-}(z,\ell), 1),(w,1) \overline{\cA(\bar z,\ell)}) = 
0
\end{equation}
for every $w \in \bbD$ and every $z \in \bbC_+$ (since $\omega_w(\bar S) \le 1$ for every $w \in \bbD$, one can use Lebesgue dominated convergence theorem and \eqref{eq14} to derive \eqref{eq006} from \eqref{eq03}). Since $\cA(z, \ell) \in \SL(2,\bbC)$, 
with $\cJ$ given in \eqref{cJmatrixdefn} the standard formula for inverse of matrices gives 
\[
\cA(z, \ell)^{-1} = \cJ \cA(z, \ell)^\top \cJ.
\]
Combining this with  \eqref{reflectionsymmetry0} we get 
\begin{equation}\label{reflectionsymmetry}
\overline{\cA(\bar z, \ell)} = j_1 \cA(z, \ell) j_1, \qquad \forall z \in \bbC,
\end{equation}
where $\overline{ (\dots)}$ denotes entry by entry complex conjugation. From \eqref{eq16} and the invariance of $\gamma$ with respect to M\"obius transformations we now see that \eqref{eq03} holds if 
\begin{equation}\label{eq19}
\lim_{\ell \to +\infty}\diam_{\gamma} D_{\cB}(z,\ell) = 0, 
\end{equation}
where $\diam_{\gamma} F= \sup_{w_1,w_2 \in F} \gamma(w_1,w_2)$ is the diameter of a set $F \subset \bbD$ with respect to $\gamma$, and
$$
D_{\cB}(z,\ell) = \{(w, 1)\cB(z,\ell)^{-1}, \; w \in \bbD\}
$$
is the Weyl disk of the $j$-monotonic family $\cB(z,\ell) = j_1 \cA^{-1}(z,\ell)j_1$. Assume first that $z = i$ and $\cA$ admits the Arov normalization condition at $z = i$:
\[
\cA(i,\ell) = \begin{pmatrix}
e^{\mu(\ell)} & 0 \\
- e^{\mu(\ell)} \kappa(\ell) & e^{-\mu(\ell)}
\end{pmatrix},
\]
so
\[
\cB(i,\ell) = \begin{pmatrix}
e^{\mu(\ell)} & e^{\mu(\ell)} \kappa(\ell)  \\
0 & e^{-\mu(\ell)}
\end{pmatrix}.
\]
The Euclidean diameter of $D_{\cB}(i,\ell)$ is equal to $2/e^{2\mu(\ell)}(1- \kappa(\ell)^2)$ by Lemma \ref{l03}. Note that we have $\mu(\ell) \to +\infty$ and $\kappa(\ell) \to s_+(i)$ as $\ell\to+\infty$ by Lemma \ref{prop18}. Since $|s_+(i)| = |s_+(i,0)|< 1$ by our assumption, we see that the Euclidean diameters of $D_{\cB}(i,\ell)$ tend to zero as $\ell \to +\infty$. Since $\cB(i,\ell)$  is upper-triangular, we have $0 \in D$. Hence, relation \eqref{eq19} holds. Moreover, \eqref{16aug1} shows that under the Arov gauge normalization we have \eqref{eq19} for every $z \in \bbC_+$.
	
\medskip
	
In the general case, Proposition \ref{l07} gives a $j$-unitary family $\{U(\ell)\}$ such that $\tilde \cA = \cA U(\ell)$ obeys the Arov normalization. Then the previous argument implies that the $\gamma$-diameters of the Weyl disks
$$
D_{\tilde \cB}(z,\ell) = \{(w, 1)\tilde \cB(z,\ell)^{-1}, \; w \in \bbD\}, \qquad 
\tilde \cB(z,\ell) = j_1\tilde \cA^{-1}(z,\ell)j_1,
$$
tend to zero. Since $\tilde \cB(z,\ell)^{-1} = \cB(z,\ell)^{-1}j_1 U(\ell) j_1$ and $j_1 U(\ell) j_1$ is $j$-unitary, the same is true for the Weyl disks $D_\cB(z,\ell)$
due to the invariance of $\gamma$ under the M\"obius transformations. This ends the proof.
\end{proof}

\section{Remling's theorem for canonical systems in the Arov form}

In this section we are interested in description of coefficients of reflectionless canonical systems. Recall that a $j$-monotonic family $\{\cA(z,\ell)\}_{\ell \in \bbR}$ in the limit point case at $\pm\infty$ is called {\it reflectionless} on a set $E \subset \bbR$ if its Schur functions $s_{\pm}$ satisfy
\begin{equation}\label{rfl}
s_{+}(x) = \overline{s_-(x)} \quad \mbox{ for almost every }\quad x \in \sE
\end{equation}
in the sense of nontangential boundary values. In the case where $\sE$ is the absolutely continuous spectrum (see the beginning of Section \ref{s6}) of  $\{\cA(z,\ell)\}_{\ell \in \bbR}$ ,  this family and the canonical system it generates are called reflectionless.

\medskip

Roughly, the main result of this section says that canonical systems with almost periodic coefficients are reflectionless. Let us introduce the notion of almost periodicity we are going to use \cite[Section 6.1]{BLY2}. Let $\mu$ be a complex Borel measure on the real line $\bbR$. The measure $\mu$ is called translation bounded if for any compact subset $K \subset \bbR$ we have
\begin{equation}
\lVert \mu \rVert_K := \sup_{x\in\bbR} \lvert \mu\rvert (x+K) < \infty.
\end{equation}
Consider a set of test functions $X \subset L^1(\mu)$ closed under translation: $S_y X \subset X$, $y \in \bbR$, where the shift operator $S_y$ is defined by
$$
S_y h: x \mapsto h(x - y), \qquad x \in \bbR.
$$
We say that a translation bounded Borel measure $\mu$ on $\bbR$ is $X$-almost periodic if for all $h \in X$, the convolution
$$
h * \mu: x \mapsto \int_{\bbR} h(x-\ell)\,d\mu(\ell). 
$$
is a uniformly almost periodic function on $\bbR$ (that is, $\{S_y(h * \mu)\}_{y \in \bbR}$ is a precompact set in the Banach space of bounded continuous functions on $\bbR$). It is classical to define almost periodic measures using the space $C_{c}(\bbR)$ of continuous functions with compact support as a set $X$ of test functions, see, e.g., \cite{ArgGil74}, \cite{ArgGil90}. We will use another class $X = PC_c(\bbR)$ of piecewise continuous compactly supported functions. It is easy to see that a translation bounded complex measure $\mu$ without point masses is $PC_c(\bbR)$-almost periodic if and only if the function
$$
x \mapsto \mu([x, x+\ell]), \qquad x \in \bbR,
$$
is uniformly almost periodic for every $\ell >0$. We call a $2 \times 2$ matrix-valued mapping $PC_c(\bbR)$-almost periodic if each its entry is $PC_c(\bbR)$-almost periodic. This allows us to deal with canonical systems with $PC_c(\bbR)$-almost periodic coefficients. Here is a reformulation of Theorem \ref{t001new} in new terms.

\begin{theorem}\label{t001diff}
Every canonical system 
$$
\cA(z,\ell) j = j + \int_0^\ell \cA(z, \xi) \left( i z A(\xi) - B(\xi)\right) d\mu(\xi), \qquad z \in \bbC,
$$
with $PC_c(\bbR)$-almost periodic coefficients $A\,d\mu$, $B\,d\mu$ is reflectionless.
\end{theorem}
Our plan, is to use the following lemma. The usage of the Breimesser-Pearson  theorem from the previous section in its proof is inspired by \cite{Rem11}.
\begin{lemma}\label{l72}
Suppose that $\{\cA(z,\ell)\}_{\ell \in \bbR}$ is a $j$-monotonic family in the limit point case at $\pm\infty$, and let $\{\tau_n\}$ be a sequence converging to $+\infty$ such that there exists the limit
$$
\cB(z,\ell) = \lim_{n \to +\infty}\cA^{[\tau_n]}(z,\ell), \qquad 
\cA^{[\tau_n]}(z,\ell) = \cA(z,\tau_n)^{-1}\cA(z,\ell+\tau_n).
$$
uniformly on compact subsets of $\bbC \times \bbR$.
Then $\{\cB(z,\ell)\}_{\ell \in \bbR}$ is a $j$-monotonic family reflectionless on the absolutely continuous spectrum $\sE$ of $\{\cA(z,\ell)\}_{\ell \in \bbR}$. 	
\end{lemma}
\begin{proof}
By definition, $\{\cB(z,\ell)\}_{\ell \in \bbR}$ is a $j$-monotonic family, and, moreover,  we have $s_{\cA,\pm}(z,\tau_n) \to s_{\cB,\pm}(z)$, $z \in \bbC_+$, for the corresponding Schur spectral functions. It follows that 
$$
\lim_{n \to \infty}\int_{e}\omega_{s_{\cA,\pm}(x,\tau_n)}(t)\,dx = \int_{e}\omega_{s_{\cB,\pm}(x)}(t)\,dx, \qquad S \subset \bbT,
$$
for every Borel set $e \subset \bbR$ of finite Lebesgue measure. Here we used again estimate \eqref{eq02} for Schur functions, see Theorem 1 in \cite{BP03} or Theorem 7.6 in \cite{Remling}. If moreover $e$ is a subset of the a.c. spectrum $\sE$ of $\{\cA(z,\ell)\}$, from  Theorem \ref{BPthm} we get
$$
\int_{e}\omega_{s_{\cB,-}(x)}(\bar S)\,dx = \int_{e}\omega_{s_{\cB,+}(x)}(S)\,dx.
$$
Using the fact that $\omega_{w}(\bar S) = \omega_{\bar w}(S)$ for every $w \in \bbD$, we conclude that $\overline{s_{\cB,-}} = s_{\cB,+}$ almost everywhere on $\sE$, that is, the system $\{\cB(z,\ell)\}_{\ell \in \bbR}$ is reflectionless on $\sE$.
\end{proof}
We are ready to prove Theorem \ref{t001diff}.

\begin{proof}
Suppose that $\cA(z,\ell)$ is the transfer matrix of a canonical system with coefficients $A\,d\mu$, $B\,d\mu$ that are $PC_c(\bbR)$-almost periodic matrix-valued measures. For $\tau \ge 0$, $n \ge 0$, let 
$$
A_{\tau}(\ell)\,d\mu_{\tau}(\ell) = A(\ell+\tau)\,d\mu_\tau, 
\quad 
B_{\tau}(\ell)\,d\mu_{\tau} = B(\ell+\tau)\,d\mu_{\tau}, 
\qquad \mu_\tau(E) = \mu(\tau+E),
$$ 
be the coefficients of the canonical system with the transfer matrix $\cA^{[\tau]}(z,\ell)  = \cA(z,\tau)^{-1}\cA(z,\ell+\tau)$. Using almost periodicity, one can choose a sequence $\{\tau_n\}$, $\lim \tau_n = +\infty$, such that  
$$
\lim_{n \to+\infty}\int_{0}^{\ell} A_{\tau_n}(t)\,d\mu_{\tau_n}(t)
=\int_{0}^{\ell} A(t)\,d\mu(t)
$$
and 
$$
\lim_{n \to+\infty}\int_{0}^{\ell} B_{\tau_n}(t)\,d\mu_{\tau_n}(t)
=\int_{0}^{\ell} B(t)\,d\mu(t)
$$
uniformly in $\ell \ge 0$ on each interval $[0, T]$. For $z \in \bbC$, let 
$X_{z,n}$ denote the matrix valued measure $-(izA_{\tau_n}-B_{\tau_n})j\,d\mu_{\tau_n}$, and let $X_z = -(izA-B)j\,d\mu$. Then
$$
\cA^{[\tau_n]}(z,\ell) = I + \int_{0}^{\ell}dX_{z,n}(t_1) + \int_{0}^{\ell}\int_{0}^{t_1} dX_{z,n}(t_2)dX_{z,n}(t_1)+  \ldots
$$
where the infinite series in the right hand side converges uniformly on compact subsets of $[0,+\infty) \times \bbC$ because $\mu$ is translation bounded. Since the variation on $[0, T]$ of $X_{z,n}-X_z$ tends to zero uniformly on compact subsets of $\bbC$ for each $T\ge 0$, we have $\cA(z,\ell) = \lim_{n \to +\infty}\cA^{[\tau_n]}(z,\ell)$ and conditions of Lemma \ref{l72} are satisfied for the sequence $\{\tau_n\}$ and $\cB = \cA$. From Lemma \ref{l72} we now see that the family $\{\cA(z,\ell)\}_{\ell \in \bbR}$ is reflectionless. 
\end{proof}

\end{document}